%% file: main.tex
\documentclass{article}
\usepackage{mathrsfs}
\usepackage[utf8]{inputenc}
\usepackage[letterpaper,margin=1.0in]{geometry}
\input{math.tex}

\theoremstyle{plain}
\newtheorem{theorem}{Theorem}[section]

\newtheorem{lemma}[theorem]{Lemma}
\newtheorem{corollary}[theorem]{Corollary}
\theoremstyle{definition}
\newtheorem{definition}[theorem]{Definition}
\newtheorem{assumption}[theorem]{Assumption}
\theoremstyle{remark}
\newtheorem{remark}[theorem]{Remark}

\usepackage{color}
\usepackage{enumitem}
\usepackage{algorithm}
\usepackage{algorithmic}
\usepackage[numbers,sort]{natbib}
\usepackage{hyperref}
\usepackage{graphicx}
\usepackage{subfigure}
\usepackage{booktabs}

\title{On the Complexity of Finite-Sum Smooth Optimization \\ under the Polyak--{\L}ojasiewicz Condition}
\date{}
\author{Yunyan Bai \qquad\qquad Yuxing Liu \qquad\qquad Luo Luo \\[0.3cm] School of Data Science, Fudan University}

\begin{document}

\maketitle

\begin{abstract}
This paper considers the optimization problem of the form $\min_{{\bf x}\in{\mathbb R}^d} f({\bf x})\triangleq \frac{1}{n}\sum_{i=1}^n f_i({\bf x})$, 
where~$f(\cdot)$ satisfies the Polyak--{\L}ojasiewicz (PL) condition with parameter $\mu$ and $\{f_i(\cdot)\}_{i=1}^n$ is $L$-mean-squared smooth.
We show that any gradient method requires at least $\Omega(n+\kappa\sqrt{n}\log(1/\epsilon))$ incremental first-order oracle (IFO) calls to find an $\epsilon$-suboptimal solution, where $\kappa\triangleq L/\mu$ is the condition number of the problem.
This result nearly matches upper bounds of IFO complexity for best-known first-order methods.
We also study the problem of minimizing the PL function in the distributed setting such that the individuals $f_1(\cdot),\dots,f_n(\cdot)$ are located on a connected network of $n$ agents.
We provide lower bounds of
$\Omega(\kappa/\sqrt{\gamma}\,\log(1/\epsilon))$, $\Omega((\kappa+\tau\kappa/\sqrt{\gamma}\,)\log(1/\epsilon))$ and $\Omega\big(n+\kappa\sqrt{n}\log(1/\epsilon)\big)$ for communication rounds, time cost and local first-order oracle calls respectively, where $\gamma\in(0,1]$ is the spectral gap of the mixing matrix associated with the network and~$\tau>0$ is the time cost of per communication round. 
Furthermore, we propose a decentralized first-order method that nearly matches above lower bounds in expectation. 
\end{abstract}

\section{Introduction}

We study the optimization problem of the form
\begin{align}\label{main:prob}
\min_{\vx\in\BR^d} f(\vx)\triangleq \frac{1}{n}\sum_{i=1}^n f_i(\vx),
\end{align}
where $\{f_i(\cdot)\}_{i=1}^n$ is $L$-mean-squared smooth but each $f_i(\cdot)$ is possibly nonconvex.
The complexity of finding stationary points in Problem~(\ref{main:prob}) has been widely studied in recent years~\cite{allen2016variance,reddi2016stochastic,zhou2019lower,fang2018spider,li2021page}.
However, finding the global solution is intractable for the general nonconvex smooth optimization~\cite{nemirovskij1983problem,roux2012stochastic}.
This paper focuses on the minimization Problem (\ref{main:prob}) under the \PL~condition (PL) \cite{polyak1963gradient,lojasiewicz1963topological}, i.e., the objective function $f:\BR^d\to\BR$ satisfies
\begin{align*}
    f(\vx) - \inf_{\vy\in\BR^d} f(\vy) \leq \frac{1}{2\mu}\Norm{\nabla f(\vx)}^2
\end{align*}
for any $\vx\in\BR^d$, where $\mu>0$ is a constant.
This inequality suggests the function value gap $ f(\vx) - f^*$ is dominated by the square of gradient norm, which leads to the gradient descent (GD) method linearly converge to the global minimum without the convexity~\cite{karimi2016linear}.
The PL condition covers a lot of popular applications, such as deep neural networks~\cite{liu2022loss,allen2019convergence,zeng2018global}, reinforcement learning~\cite{fazel2018global,agarwal2021theory,mei2020global,yuan2022general}, optimal control~\cite{bu2019lqr,fatkhullin2021optimizing} and matrix recovery~\cite{hardt2016identity,li2018algorithmic,bi2022local}.

We typically solve the finite-sum optimization Problem (\ref{main:prob}) by incremental first-order oracle (IFO) methods \cite{agarwal2015lower}, which can access the pair $\{f_i(\vx), \nabla f_i(\vx)\}$ for given $\vx\in\BR^d$ and $i\in[n]$.
This class of methods can leverage the structure of the objective \cite{johnson2013accelerating,defazio2014saga,schmidt2017minimizing,zhang2013linear} to iterate with one or mini-batch individual gradient, which is more efficient than the iteration with the full-batch gradient. 
IFO methods have received a lot of attention in recent years.
For example, the stochastic variance reduced gradient (SVRG) methods with negative momentum \cite{qian2021svrg,allen2017katyusha,kovalev2020don,woodworth2016tight,agarwal2015lower} achieve the (near) optimal IFO complexity for convex optimization;
the stochastic recursive gradient methods \cite{nguyen2017sarah,fang2018spider,li2021page,pham2020proxsarah,wang2019spiderboost,zhou2019lower,carmon2020lower}
achieve the optimal IFO complexity for finding approximate stationary points in general nonconvex optimization.
For the PL condition, \citet{reddi2016stochastic,lei2017non} proposed SVRG-type methods that find the $\epsilon$-suboptimal solution within at most~$\fO((n+n^{2/3}\kappa)\log(1/\epsilon))$ IFO calls, where~$\kappa\triangleq L/\mu$ is the condition number.
Later, \citet{zhou2019faster,wang2019spiderboost,li2021page} improved the upper bound to $\fO((n+\kappa\sqrt{n})\log(1/\epsilon))$ by stochastic recursive gradient estimator.
Recently, \citet{yue2023lower} established a tight lower complexity bound of the full-batch gradient methods for minimizing the PL function.
However, the optimality of existing IFO methods for the finite-sum setting is still an open problem.

For large-scale optimization problems, we are interested in designing the distributed algorithms.
Specifically, we allocate individuals $f_1(\cdot),\dots,f_n(\cdot)$ on~$n$ different agents and desire the agents solve the problem collaboratively.
We focus on the decentralized setting that agents are linked by a connected network, so that each agent can only access its own local first-order oracle (LFO) and exchange messages with its neighbours.
Besides the LFO complexity, we also require considering the communication complexity and the time complexity.
It is worth noting that the time complexity in distributed optimization does not directly corresponds to the weighted sum of the LFO complexity and the communication complexity, since some agents may skip the computation of local gradient during the iterations \cite{maranjyan2022gradskip,mishchenko2022proxskip}.
Most of work for decentralized optimization focus on the convex case \cite{shi2015extra,nedic2009distributed,qu2017harnessing,scaman2017optimal,kovalev2020optimal,song2023optimal,ye2023multi,hendrikx2021optimal,li2022variance} or the general nonconvex case~\cite{luo2022optimal,li2022destress,xin2022fast,sun2020improving,lu2021optimal,zhan2022efficient}.
Recently, \citet{yuan2022revisiting} studied the tightness of complexity for decentralized optimization under the PL condition in online setting, but the optimality of their result does not include the dependence on condition number.

In this paper, we provide the nearly tight lower bounds for the finite-sum optimization problem under the PL condition.
We summarize our contributions as follows:
\begin{itemize}[topsep=3pt,itemsep=2pt,partopsep=2pt, parsep=2pt]
\item We provide the lower bound of $\Omega(n+\kappa\sqrt{n}\log(1/\epsilon))$ for IFO complexity, which nearly match the upper complexity bound of $\fO((n+\kappa\sqrt{n})\log(1/\epsilon))$ \cite{zhou2019faster,wang2019spiderboost,li2021page}.
\item We provide the lower bounds of $\Omega(\kappa/\sqrt{\gamma}\,\log(1/\epsilon))$, $\Omega((\kappa+\tau\kappa/\sqrt{\gamma}\,)\log(1/\epsilon))$ and $\Omega\big(n+\kappa\sqrt{n}\log(1/\epsilon)\big)$ for communication complexity, time complexity and LFO complexity for decentralized setting, where $\gamma$ is the spectral gap of the mixing matrix for the network and~$\tau$ is the time cost of per communication round. 
\item We propose a decentralized first-order algorithm within  communication complexity of $\tilde \fO(\kappa/\sqrt{\gamma}\,\log(1/\epsilon))$, time complexity of $\tilde \fO((\kappa+\tau\kappa/\sqrt{\gamma}\,)\log(1/\epsilon))$ and the LFO complexity of $\fO\big((n+\kappa\sqrt{n}\,)\log(1/\epsilon)\big)$ in expectation, nearly matching the lower bounds.
\end{itemize}
We compare our results with related work in Table~\ref{table:single}-\ref{table:distributed}.

\section{Preliminaries}

In this section, we formalize the problem setting and the complexity of the finite-sum optimization.

\subsection{Notation and Assumptions}

Given vector $\vx\in\BR^d$, we denote $x_i$ as the $i$-th entry of~$\vx$ for~$i\in[d]$ and denote ${\rm supp}(\vx)$ as the index set for nonzero entries of $\vx$. 
Given matrix $\mA\in\BR^{m\times n}$, we denote~$a_{i,j}$ as the $(i,j)$-th entry of $\mA$ for $i\in[m]$ and $j\in[n]$. 
We let $\vone$ be the vectors (or matrices) of all ones and $\vzero$ be the vector (or matrix) of all zeros.
Additionally, we let $\mI$ be the identity matrix and denote its $i$-th column as $\ve_i$. 
We use $\Vert \cdot \Vert$ to present the Euclidean norm of a vector or the Frobenius norm of a matrix.

\begin{table*}[t]
\centering
\caption{We present the IFO complexity for finding an $\epsilon$-suboptimal solution of Problem (\ref{main:prob}) on single machine.}\label{table:single}
\vskip0.3cm 
\begin{tabular}{ccc}
\hline
 Algorithm & IFO  & Reference \\
\hline\hline\addlinespace
GD & $\fO\big(\kappa n\log(1/\epsilon)\big)$  & \citet{karimi2016linear} \\\addlinespace
SVRG / SCSG & $\fO\big((n+\kappa n^{2/3}\,)\log(1/\epsilon)\big)$ &  \citet{reddi2016stochastic,lei2017non}
\\\addlinespace
SPIDER / PAGE   & $\fO\big((n+\kappa\sqrt{n}\,)\log(1/\epsilon)\big)$ & \citet{zhou2019faster,wang2019spiderboost,li2021page}  \\ \addlinespace
\hline\addlinespace
Lower Bound & $\Omega\big(n+\kappa\sqrt{n}\log(1/\epsilon)\big)$ & Corollary \ref{cor:ifo}  \\ \addlinespace \hline
\end{tabular}  
\end{table*}

We consider the following assumptions for the finite-sum optimization Problem (\ref{main:prob}).
\begin{assumption}\label{asm:lower}
We suppose the objective function $f(\cdot)$ is lower bounded, i.e, we have
\begin{align*}
    f^* = \inf_{\vy\in\BR^d} f(\vy) > -\infty.
\end{align*} 
\end{assumption}
\begin{assumption}\label{asm:as}
We suppose the function set $\{f_i(\cdot)\}_{i=1}^n$ is $L$-mean-squared smooth for some $L>0$, i.e., we have
\begin{align*}
 \frac{1}{n}\sum_{i=1}^n\Norm{\nabla f_i(\vx)-\nabla f_i(\vy)}^2\leq L^2\Norm{\vx-\vy}^2,
\end{align*}
for any $\vx,\vy\in\BR^d$. 
\end{assumption}
\begin{assumption} \label{asm:SC}
We suppose objective function $f(\cdot)$ is $\mu$-PL for some $\mu>0$, i.e., 
we have
\begin{align*}
    f(\vx) -\inf_{\vy\in\BR^d} f(\vy) \leq \frac{1}{2\mu}\norm{\nabla f(\vx)}^2
\end{align*}
for any $\vx \in \BR^{d}$.
\end{assumption}
Based on above assumptions, we define condition number and $\epsilon$-suboptimal solution of our problem.
\begin{definition}
We define $\kappa\triangleq L/\mu$ as the condition number of problem~(\ref{main:prob}).
\end{definition}

\begin{definition}
We say $\hat \vx$ is an $\epsilon$-suboptimal solution of Problem (\ref{main:prob}) if it holds that $f(\hat \vx) - f^* \le \epsilon$.
\end{definition}

We use the notation $\tilde\fO(\cdot)$ to hide the logarithmic dependence on condition number $\kappa$ and individuals number $n$.

For the decentralized setting, the individual $f_i(\cdot)$  presents the local function on the $i$-th agent.
We define the aggregate variable and the corresponding aggregated gradient as
\begin{align*}
\begin{split}    
\mX=\begin{bmatrix}
\vx(1) \\ \vdots \\ \vx(n)
\end{bmatrix}\in\BR^{n\times d}
\qquad\text{and}\qquad
\nabla \mF(\mX)=\begin{bmatrix}
\nabla f_1(\vx(1)) \\ \vdots \\ \nabla f_n(\vx(n))
\end{bmatrix}
\in\BR^{n\times d}
\end{split}
\end{align*}
respectively, where $\vx(i)\in\BR^{1\times d}$ is the local variable on the $i$-th agent. 
For given $\mX\in\BR^{n\times d}$, we also introduce the mean vector $\bx = \frac{1}{n}\vone^\top\mX = \frac{1}{n}\sum_{i=1}^n \vx(i)\in\BR^{1\times d}$.
For the ease of presentation, we let the input of a function can also be organized as a row vector, such as $f(\bar\vx)$ and $f_i(\vx(i))$.

We describe one communication round by multiplying the mixing matrix $\mW\in\BR^{n\times n}$ on the aggregated variable.
We give the following assumption for matrix $\mW$.

\begin{assumption}\label{asm:W}
We suppose mixing matrix $\mW\in\BR^{n\times n}$ has the following properties:
(a) We have $\mW=\mW^\top$, $\mW\vone_n = \vone_n$ and $\vzero \preceq \mW \preceq \mI$;
(b) The entry of $\mW$ holds that~$w_{i,j}>0$ if and only if the $i$-th agent and the $j$-th agent are connected or $i=j$, otherwise it holds that $w_{ij}=0$.
(c) The spectral gap of $\mW$ is lower bounded by $\gamma\in(0,1]$, i.e., it holds that $1-\lambda_2(\mW)\geq\gamma$ for some $\gamma\in(0,1]$, where $\lambda_2(\mW)$ is the second-largest eigenvalue of $\mW$.
\end{assumption}

\subsection{The Finite-Sum Optimization}

The complexity of first-order methods for solving the finite-sum optimization Problem (\ref{main:prob}) on single machine mainly depends on the number of access to the incremental first-order oracle (IFO), which is defined as follows \cite{agarwal2015lower,woodworth2016tight}.

\begin{definition}
The incremental first-order oracle (IFO) takes the input $i\in[n]$ and $\vx$, and returns the pair $(f_i(\vx),\nabla f_i(\vx))$.
\end{definition}

Then we formally define the IFO algorithm.

\begin{definition}\label{def:linear-span}
An IFO algorithm for given initial point~$\vx^0$ is defined as a measurable mapping from functions $\{f_i(\cdot)\}_{i=1}^n$ to an infinite sequence of point and index
pairs $\{(\vx^t, i_t)\}_{t=0}^\infty$ with random variable $i_t\in[n]$, which satisfies
\begin{align*}
\vx^t \in \mathrm{Lin}\big(\{\vx^0, \ldots, \vx^{t-1}, 
\nabla f_{i_0}(\vx^0), \ldots, \nabla f_{i_{t-1}}(\vx^{t-1})\}\big),
\end{align*}
where $\mathrm{Lin}(\cdot)$ denotes the linear span and $i_t$ denotes the index of individual function chosen at the $t$-th step. 
\end{definition}

For the distributed optimization over a network of $n$ agents, the $i$-th agent can only perform the computation on its local function $f_i(\cdot)$ directly.
Hence, we describe the complexity of computational cost by the number of access to the local first-order oracle (LFO).
\begin{definition}
The local first-order oracle (IFO) takes the input $i\in[n]$ and $\vx$ , and returns the pair $(f_i(\vx), \nabla f_i(\vx))$.
\end{definition}

Recall that agents on network can only communicate with their neighbours, which means the agent in decentralized algorithms cannot arbitrarily establish the linear space of all local gradients. 
Additionally, one iteration of the algorithm allows a mini-batch of agents to compute their local gradient in parallel.
Therefore, besides the LFO complexity, we also need to separately consider the communication complexity and the time complexity.
This motivates the following definition for decentralized first-order oracle algorithm (DFO) \cite{scaman2017optimal}.

\begin{definition}\label{dfn:DFO}
A decentralized first-order oracle (DFO) algorithm over a network of $n$ agents satisfies the following constraints:
\begin{itemize}
\item \textbf{Local memory:} Each agent $i$ can store past values in a local memory $\mathcal{M}_{i,s}$ at time $s>0$. 
These values can be accessed and used at time $s$ by running the algorithm on agent $i$. Additionally, for all $i\in[n]$, we have 
\begin{align*}
\mathcal{M}_i^s\subset\mathcal{M}_{{\rm comp},i}^s\bigcup\mathcal{M}_{{\rm comm},i}^s,
\end{align*}
where $\mathcal{M}_{{\rm comp},i}^s$ and $\mathcal{M}_{{\rm comm},i}^s$ are the values come from the computation and communication respectively. 
\item \textbf{Local computation:} Each agent $i$ can access its local first-order oracle $\{f_i(\vx), \nabla f_i(\vx)\}$ for given $\vx\in\mathcal{M}_{i,s}$ at time $s$. That is, for all $i\in[n]$, we have
\begin{align*}
\mathcal{M}_{{\rm comp},i}^s={\rm Lin}\big(\{\vx,\nabla f_i(\vx):x\in\mathcal{M}_i^{s-1}
\}\big).
\end{align*}
\item \textbf{Local communication:} Each agent $i$ can share its value to all or part of its neighbors at time $s$. That is, for all $i\in[n]$, we have
\begin{align*}
\mathcal{M}_{{\rm comm},i}^{s}={\rm Lin}\Bigg(\bigcup_{j\in{\rm nbr}(i)}\mathcal{M}_j^{s-\tau}\Bigg),
\end{align*}
where ${\rm nbr}(i)$ is the set consists of the indices for the neighbours of agent $i$ and $\tau<s$.
\item \textbf{Output value:} Each agent $i$ can specify one vector in its memory as local output of the algorithm at time $s$. That is, for all $i\in[n]$, we have
$\vx_i^s\in\mathcal{M}_i^s$.     
\end{itemize}
\end{definition}

\begin{table*}[t]
\centering
\footnotesize
\caption{We present the complexity for finding an $\epsilon$-suboptimal solution of Problem (\ref{main:prob}) in decentralized setting.} \label{table:distributed}
\vskip0.3cm 
\renewcommand{\arraystretch}{1.2}
\begin{tabular}{ccccc}
\hline
Algorithm  & Communication  & Time  & LFO & Reference \\
\hline\hline\addlinespace
DGD-GT     & $ \tilde\fO\left({\kappa}/{\sqrt{\gamma}}\log(1/\epsilon)\right)$ & $\tilde\fO\left(\kappa\left(1+{\tau}/{\sqrt{\gamma}\,}\right)\log(1/\epsilon)\right)$ & $\fO\big(\kappa n\log(1/\epsilon)\big)$ & Theorem 
 \ref{thm:DGD-GT} \\\addlinespace
DRONE    & $ \tilde\fO\left({\kappa}/{\sqrt{\gamma}}\log(1/\epsilon)\right)$ & $\tilde\fO\left(\kappa\left(1+{\tau}/{\sqrt{\gamma}}\,\right)\log(1/\epsilon)\right)$ & $\fO\big((n+\kappa\sqrt{n}\,)\log(1/\epsilon)\big)$ & Corollary \ref{cor:decentralized-upper} \\\addlinespace 
\hline\addlinespace
Lower Bound    & $ \Omega\left({\kappa}/{\sqrt{\gamma}}\log(1/\epsilon)\right)$ & $\Omega\left(\kappa\left(1+{\tau}/{\sqrt{\gamma}}\,\right)\log(1/\epsilon)\right)$ & $\Omega\big(n+\kappa\sqrt{n}\log(1/\epsilon)\big)$ & Theorem 
 \ref{thm:decentralized-lower}, Corollary \ref{cor:decentralized-lfo} \\\addlinespace
\hline
\end{tabular} 
\end{table*}

\section{The Lower Bound on IFO Complexity}\label{sec:IFO}

This section provides the lower bound on IFO complexity to show the optimality (up to logarithmic factors) of existing first-order methods~\cite{zhou2019faster,wang2019spiderboost,li2021page}.
Without loss of generality, we always assume the IFO algorithm iterates with the initial point $\vx^{(0)}=\vzero$ in our analysis for lower bound. Otherwise, we can take the functions $\{f_i(\vx+\vx^{(0)})\}_{i=1}^n$ into consideration.

We first consider the case of $n=\fO(\kappa^2)$. 
We introduce the functions $\psi_\theta:\BR\to\BR$, $q_{T,t}:\BR^{Tt}\to\BR$ and $g_{T,t}:\BR^{Tt}\to\BR$ provided by \citet{yue2023lower}, that is
\begin{align*}   
& \psi_{\theta}(x)
= \begin{cases}
	\frac{1}{2}x^{2}, 
	& x\leq\frac{31}{32}\theta, \\[0.3em]
	\frac{1}{2}x^{2}-16(x-\frac{31}{32}\theta)^{2},& {\frac{31}{32}\theta<x\leq \theta,}\\[0.3em]
	\frac{1}{2}x^{2}-\frac{1}{32}\theta^2+16(x-\frac{33}{32}\theta)^{2}, &      
	{\theta<x\leq \frac{33}{32}\theta,}\\[0.3em]
	\frac{1}{2}x^{2}-\frac{1}{32}\theta^2,      
	& {x>\frac{33}{32}\theta,}
\end{cases} \\[0.15cm]
& q_{T,t}(\vx)=\frac{1}{2}\sum_{i=0}^{t-1}\Big(\Big(\frac{7}{8}x_{iT} -x_{iT+1}\Big)^2+\sum_{j=1}^{T-1}(x_{iT+j+1}-x_{iT+j})^2\Big), \\
& \text{and}\quad g_{T,t}(\vx)=q_{T,t}(\vb-\vx)+\sum_{i=1}^{Tt}\psi_{b_i}(b_i-x_i), 
\end{align*}
where we define $x_0=0$ and $\vb\in \mathbb{R}^{Tt}$ with $b_{kT+\tau}=({7}/{8})^k$ for~$k\in \{0\}\cup[T-1]$ and $\tau\in[T]$.
We can verify that 
\begin{align*}
g_{T,t}^*\triangleq \inf_{\vy\in\BR^{Tt}}g_{T,t}(\vy)=0.    
\end{align*}
The following lemma shows the function $g_{T,t}$ holds the zero-chain property \cite{nesterov2018lectures,carmon2020first} and describes its smoothness, PL parameter and optimal function value gap, which results the tight lower bound of full-batch first-order methods \citep[Section 4]{yue2023lower}.
\begin{lemma}\label{lem:basic function} 
The function $g_{T,t}:\BR^{Tt}\to\BR$ holds that:
\begin{enumerate}[topsep=0pt,itemsep=2.5pt,partopsep=2.5pt, parsep=2.5pt]
\item[(a)]\label{con:zero-chain} For any $\vx\in\BR^{Tt}$ satisfying ${\rm supp}(\vx)\subseteq\left\{1,2,\cdots,k\right\}$, it holds ${\rm supp}(\nabla g_{T,t}(\vx))\subseteq\left\{1,2,\cdots,k+1\right\}$.
\item[(b)]\label{con:smooth} The function $g_{T,t}$ is 37-smooth.
\item[(c)]\label{con:PL} The function $g_{T,t}$ is $1/(aT)$-PL with $a=19708$.  
\item[(d)] The function $g_{T,t}$ satisfies that $g_{T,t}(\mathbf{0})-g_{T,t}^*\leq3T$. 
\item[(e)] For any  $\delta<0.01$, $t=2\lfloor\log_{{8}/{7}}{2}/{(3\delta)}\rfloor$ and $\vx\in\BR^{Tt}$ satisfying ${\rm supp}(\vx)\subseteq\left\{1,2,\cdots,Tt/2\right\}$, it holds that $g_{T,t}(\vx)-g_{T,t}^*> 3T\delta$. 
\end{enumerate}
\end{lemma}

We can establish the mean-squared smooth functions by the composition of orthogonal transform.
Compared with the study on convex and general nonconvex problem~\cite{carmon2020lower,zhou2019lower}, we present the following lemma by additionally considering the PL condition.
\begin{lemma}\label{lem:finite-sum}
Given a function $g:\mathbb{R}^m\rightarrow\mathbb{R}$ that is $\hat{L}$-smooth and $\hat{\mu}$-PL, define $f_i(\vx)=g(\mU^{(i)}\vx)$ with $\vx\in\BR^{mn}$, $i\in[n]$ and $\mU^{\left(i\right)}=[\ve_{(i-1)m+1},\cdots,\ve_{im}]^\top\in\BR^{m\times mn}$.
Then the function set $\{f_i:\mathbb{R}^{mn}\rightarrow\mathbb{R}\}_{i=1}^n$ is $\hat{L}/\sqrt{n}$-mean-squared smooth, 
and the function $f(\cdot)=\frac{1}{n}\sum_{i=1}^{n}f_i(\cdot)$ is~$\hat{\mu}/n$-PL with $f(\mathbf{0})-\inf_{\vy\in\BR}f(\vy)=g(\mathbf{0})-\inf_{\vy\in\BR}g(\vy)$.  
\end{lemma}

To achieve the hard instance functions with the desired smoothness and PL parameters, we also require the scaling lemma as follows.
\begin{lemma}\label{lem:scale}
	Suppose the function $g:\mathbb{R}^m\rightarrow\mathbb{R}$ is $\hat{L}$-smooth, $\hat{\mu}$-PL and has lower bound $g^*=\inf_{\vy\in\BR^m} g(\vy)$, 
    then the function $\hat g(\vx)=\alpha g(\beta \vx)$ is $\alpha\beta^2\hat{L}$-smooth, $\alpha\beta^2\hat{\mu}$-PL and satisfies that $\hat g(\mathbf{0})-\hat g^*=\alpha (g(\mathbf{0})-g^*)$ for any $\alpha,\beta>0$, where $\hat g^*=\inf_{\vy\in\BR^m} \hat g(\vy)$.
\end{lemma}

Based on above lemmas \ref{lem:basic function}, \ref{lem:finite-sum} and \ref{lem:scale}, we provide the lower bound of~$\Omega\big(\kappa\sqrt{n}\log(1/\epsilon)\big)$ on IFO complexity for large~$\kappa$.

\begin{theorem}\label{thm:kn}
For any $L,\mu,n,\Delta$ and $\epsilon$ with $\epsilon<0.005\Delta$ and
$L\geq37a\sqrt{n}\mu$, 
there exists $L$-mean-squared smooth function set $\{f_i:\mathbb{R}^{d}\to\BR\}_{i=1}^n$ 
with $d=\fO(\kappa\sqrt{n}\log(1/\epsilon))$ 
such that the function $f(\cdot)=\frac{1}{n}\sum_{i=1}^{n}f_i(\cdot)$ is $\mu$-PL with~$f(\vx^0)-f^*\leq\Delta$. In order to find an $\epsilon$-suboptimal solution of problem $\min_{\vx\in\BR^d}f(\vx)$, any IFO algorithm needs at least $\Omega\big(\kappa\sqrt{n}\log(1/\epsilon)\big)$ IFO calls.
\end{theorem}

We present the proof sketch of Theorem \ref{thm:kn} as follows and defer the details in Appendix \ref{sec:thm:kn}.
Specifically, we take the function $g_{T,t}:\BR^{Tt}\to\BR$ with $T=\lfloor L/(37a\sqrt{n}\mu)\rfloor$ and $t=2\lfloor\log_{{8}/{7}}{\Delta}/{(3\epsilon)}\rfloor$, and let $\hat g(\vx)=\alpha g_{T,t}(\beta\vx)$ with $\alpha=\Delta/(3T)$ and $\beta=\sqrt{3\sqrt{n}LT/(37\Delta)}$.
We also define function set $\{f_i:\BR^{nTt}\to\BR\}_{i=1}^n$ by following Lemma~\ref{lem:finite-sum} with $g(\vx)=\hat g(\vx)$ and $m=Tt$.
Then statements (b)-(d) of Lemma \ref{lem:basic function} and the scaling property shown in Lemma~\ref{lem:scale} means such construction results that the condition number of the problem is $\kappa=L/\mu$ and the optimal function value is gap $\Delta$.
Finally, the statements (a) and (e) of Lemma~\ref{lem:basic function} leads to the desired lower complexity bound.

We then consider the case of $n=\Omega(\kappa^2)$. 
Following the analysis by \citet[Theorem 2]{li2021page}, we establish the lower bound of $\Omega(n)$ for IFO complexity by introducing $\{f_i:\BR^d\to\BR\}_{i=1}^n$ 
with~$f_i(\vx)=c\left< \mathbf{u}_i, \vx \right>+\frac{L}{2}\Vert \vx\Vert^2$, where 
$\mathbf{u}_i=[\BI(\lceil{1}/{(2n)}\rceil=i),\cdots,\BI(\lceil{2n^2}/{(2n)}\rceil=i)]^\top\in\mathbb{R}^{2n^2}$, 
$d=2n^2$, $c=\sqrt{L\Delta}$ and $\BI(\cdot)$ is the indicator function.
We can verify $\{f_i:\BR^d\to\BR\}_{i=1}^n$ is $L$-mean-squared smooth and $f(\cdot)=\frac{1}{n}\sum_{i=1}^nf_i(\cdot)$ is $\mu$-PL with the optimal function value gap~$\Delta$.
Additionally, each $f_i(\cdot)$ holds zero-chain property, which leads to the lower bound of $\Omega(n)$ on IFO complexity.
We formally present this result in the following theorem and leave the detailed proof in Appendix \ref{sec:thm:n}.
\begin{theorem}\label{thm:n}
  For any $L,\mu,n,\Delta$ and $\epsilon$ with $\epsilon<{\Delta}/{2}$ and~$L\geq\mu$, there exists $L$-mean-squared smooth function set $\{f_i:\BR^d\to\BR\}_{i=1}^n$ with~$d=\fO(n^2)$ such that the function $f(\cdot)=\frac{1}{n}\sum_{i=1}^{n}f_i(\cdot)$ is $\mu$-PL with $f(\vx^0)-f^*\leq \Delta$. In order to find an $\epsilon$-suboptimal solution of problem $\min_{\vx\in\BR^d}f(\vx)$, any IFO algorithm needs at least~$\Omega(n)$ IFO calls.
\end{theorem}

We combine Theorem \ref{thm:kn} and \ref{thm:n} to achieve the lower bound on IFO complexity for the first-order finite-sum optimization under the PL condition.
\begin{corollary}\label{cor:ifo}
    For any $L,\mu,n,\Delta$ and $\epsilon$ with $\epsilon<0.005\Delta$ and $L\geq \mu$.
    there exists $L$-mean-squared smooth function set $\{f_i:\mathbb{R}^{d}\rightarrow\mathbb{R}\}_{i=1}^n$ with $d=\fO(n^2+\kappa\sqrt{n}\log(1/\epsilon)$
    such that the function $f(\cdot)=\frac{1}{n}\sum_{i=1}^{n}f_i(\cdot)$ is $\mu$-PL and satisfies $f(\vx^0)-f^*\leq\Delta$. 
    In order to find an $\epsilon$-suboptimal solution of problem $\min_{\vx\in\BR^d}f(\vx)$, any IFO algorithm needs at least $\Omega\big(n+\kappa\sqrt{n}\log(1/\epsilon)\big)$ IFO calls.
\end{corollary}

Noticing that the lower bound of $\Omega\big(n+\kappa\sqrt{n}\log(1/\epsilon)\big)$ on IFO complexity shown in Corollary \ref{cor:ifo} nearly matches the upper bound of $\fO\big((n+\kappa\sqrt{n}\,)\log(1/\epsilon)\big)$ achieved by stochastic recursive gradient algorithms~\cite{zhou2019faster,wang2019spiderboost,li2021page}.

\section{Lower Bounds in Decentralized Setting}\label{sec:decentralized}
This section provides lower bounds for decentralized setting. 
The main idea of our construction is splitting the function $g_{T,t}:\BR^{Tt}\to\BR$ (defined in Section \ref{sec:IFO}) as follows
\begin{align*}
    g_{T,t}(\vx)=q_1(\vb-\vx)+q_2(\vb-\vx)+r(\vx),
\end{align*}
where we define the functions $q_1:\BR^{Tt}\to\BR$, $q_2:\BR^{Tt}\to\BR$ and $r:\BR^{Tt}\to\BR$ as 
\begin{align*}
\begin{split}    
&	q_1(\vx)=\frac{1}{2}\sum_{i=1}^{Tt/2}(x_{2i-1}-x_{2i})^2,  \\
&	q_2(\vx)=\frac{1}{2}\sum_{i=0}^{t-1}\left[\Big(\frac{7}{8}x_{iT}-x_{iT+1}\Big)^2+\sum_{j=iT/2+1}^{(i+1)T/2-1} (x_{2j}-x_{2j+1})^2\right],  \\
& \text{and} \quad r(\vx)=\sum_{i=1}^{Tt}\psi_{b_i}(b_i-x_i),
\end{split}
\end{align*}

Based on above decomposition for $g_{T,t}$, we can establish an hard instance of $n$ individual functions for communication complexity and time complexity.
We let $\fG=\{\fV,\fE\}$ be the graph associated to the network  of the agents, where the node set $\fV=\{1,\dots,n\}$ corresponds to the $n$ agents and the edge set $\fE=\{(i,j):\text{node $i$ and node $j$ are connected}\}$ describes the topology for the network of agents.  
For given a subset $\fC\subseteq\mathcal{V}$, we define the function $h_i^\fC(\vx):\BR^{Tt}\rightarrow\BR$~as
\begin{align*}
h_i^\fC(\vx)=
\begin{cases}
\frac{r(\vx)}{n}+\frac{q_1(\vb-\vx)}{\vert \fC\vert},   &{i\in \fC}, \\[0.4em]
\frac{r(\vx)}{n}+\frac{q_2(\vb-\vx)}{\vert \fC_{\sigma}\vert},   &{i\in \fC_{\sigma}}, \\[0.4em]
\frac{r(\vx)}{n},   &{\text{otherwise}},
\end{cases}
\end{align*}
where $\fC_{\sigma}=\left\{v\in\mathcal{V}:{\rm dis}(\fC,v)\geq \sigma\right\}$ and ${\rm dis}(\fC,v)$ is the distance between set $\fC$ and node $v$.

We can verify the function sets $\{h_i^\fC:\BR^{Tt}\rightarrow\BR\}_{i=1}^n$ has the following properties.
\begin{lemma}\label{lem:decentralized instance}
We define $h(\cdot)=\frac{1}{n}\sum_{i=1}^{n}h_i^\fC(\cdot)$, then we have:
\begin{enumerate}[topsep=0pt,itemsep=1.5pt,partopsep=1.5pt, parsep=1.5pt]
\item[(a)] The function sets $\{h_i^\fC:\BR^{Tt}\to\BR\}_{i=1}^n$ is mean-squared smooth with parameter $33/n+\max\left\{2/\vert \fC\vert,2/\vert \fC_{\sigma}\vert\right\}$.
\item[(b)] The function $h:\BR^{Tt}\to\BR$ is $1/\left(anT\right)$-PL.
\item[(c)] The function holds $h(\mathbf{0})-\inf_{\vy\in\BR^{Tt}}h(\vy)\leq3T/n$.
\end{enumerate}
\end{lemma}

Now we provide the lower bounds for communication complexity and time complexity based on the scaling on the functions $\{h_i^\fC:\BR^{Tt}\rightarrow\BR\}_{i=1}^n$.
\begin{lemma}\label{lem:communication time}
	We let $f_i(\vx)=\alpha h_i^\fC(\beta \vx)$ with $\vx\in\BR^{Tt}$ and $t=2\lfloor\log_{{8}/{7}}({2}/{3\delta})\rfloor$ for any $i\in[n]$ and some  $\alpha,\beta>0$. 
    For given $\delta<0.01$, any DFO algorithm takes at least $Tt\sigma/2$ communication complexity and $Tt(1+\sigma\tau)/2$ time complexity to achieve an $3\alpha T\delta/n$-suboptimal solution of problem $\min_{\vx\in\BR^{Tt}}\frac{1}{n}\sum_{i=1}^n f_i(\vx)$.
\end{lemma}

For given spectral gap $\gamma\in(0,1]$, we consider the linear graph $\fG=\{\fV,\fE\}$ with the node set $\fV=\{1,\dots,n\}$ and the edge set $\fE=\{(i,j):|i-j|=1, i\in\fV ~~\text{and}~~ j\in\fV\}$.
Combining with the function $f_i(\vx)=\alpha h_i^\fC(\beta \vx)$ defined in Lemma \ref{lem:communication time}, we achieve the lower bounds of communication complexity and time complexity as follows.

\begin{theorem}\label{thm:decentralized-lower}
For any $L,\mu,\Delta,\gamma$ and $\epsilon$ with $L\geq194a\mu$, $\gamma\in(0,1]$ and $\epsilon<0.01\Delta$, there exist matrix $\mW\in\BR^{n\times n}$ with $1-\lambda_2(\mW)\geq\gamma$ and $L$-mean-squared smooth function set $\{f_i:\mathbb{R}^{d}\rightarrow\mathbb{R}\}_{i=1}^n$ with $d=\fO(\kappa\log(1/\epsilon))$ such that $f(\cdot)=\frac{1}{n}\sum_{i=1}^{n}f_i(\cdot)$ is $\mu$-PL with $f(\mathbf{0})-f^*\leq\Delta$ and $\gamma(\mW)=\gamma$. 
In order to find an $\epsilon$-suboptimal solution of problem 
$\min_{\vx\in\BR^d}f(\vx)$, 
any DFO algorithm needs at least $\Omega\big(\kappa/\sqrt{\gamma}
\log(1/\epsilon)\big)$ communication rounds and $\Omega\big(\kappa\log(1/\epsilon)(1+\tau/\sqrt{\gamma})\big)$ time cost.
\end{theorem}

The lower complexity bound on LFO complexity for decentralized setting can be achieved by applying Corollary \ref{cor:ifo} on fully connected network.

\begin{corollary}\label{cor:decentralized-lfo}
For any $L,\mu,n,\Delta,\gamma$ and $\epsilon$ with $\epsilon<0.005\Delta$, 
$L\geq\mu$ and $\gamma\in(0,1]$ , 
there exist mixing matrix $\mW\in\BR^{n\times n}$ with $1-\lambda_2(\mW)\geq\gamma$ and $L$-mean-squared smooth function set $\{f_i:\mathbb{R}^{d}\rightarrow\mathbb{R}\}_{i=1}^n$ with $d=\fO(n^2+\kappa\sqrt{n}\log(1/\epsilon)$ such that the function $f(\cdot)=\frac{1}{n}\sum_{i=1}^{n}f_i(\cdot)$ is $\mu$-PL and satisfies $f(\vx^0)-f^*\leq\Delta$. 
In order to find an $\epsilon$-suboptimal solution of problem $\min_{\vx\in\BR^d}f(\vx)$, any DFO algorithm needs at least $\Omega\big(n+\kappa\sqrt{n}\log(1/\epsilon)\big)$ LFO calls.
\end{corollary}

\section{Decentralized First-Order Algorithms}\label{sec:algorithm}

\begin{algorithm}[t]
\caption{GD} \label{alg:gd}
\begin{algorithmic}[1]
	\STATE \textbf{Input:} initial point $\vx^0\in\BR^d$, iteration number $T$ and stepsize $\eta>0$  \\[0.12cm]
	\STATE \textbf{for} $t = 0, 1, \dots, T-1$ \textbf{do}\\[0.12cm]
	\STATE\quad $\vx^{t+1} = \vx^t - \eta \nabla f(\vx^t)$ \\[0.05cm] 
	\STATE\textbf{end for} \\[0.05cm]
	\STATE \textbf{Output:} $\vx^T$ 
\end{algorithmic}
\end{algorithm}

\begin{algorithm}[t]
\caption{$\FM(\mY^0,\mW, K)$} \label{alg:acc-gossip}
\begin{algorithmic}[1]
	\STATE \textbf{Initialize:} $\mY^{-1}=\mY^{0}$ \\[0.12cm]
	\STATE $\eta_y=1/\big(1+\sqrt{1-\lambda_2^2(\mW)}\,\big)$ \\[0.12cm]
	\STATE \textbf{for} $k = 0, 1, \dots, K$ \textbf{do}\\[0.12cm]
	\STATE\quad $\mY^{k+1}=(1+\eta_y)\mW\mY^{k}-\eta_y\mY^{k-1}$ \\[0.12cm] 
	\STATE\textbf{end for} \\[0.12cm]
	\STATE \textbf{Output:} $\mY^K$ 
\end{algorithmic}
\end{algorithm} 

\begin{algorithm}[t]
\caption{DGD-GT} \label{alg:DGD-GT}
\begin{algorithmic}[1]
\STATE \textbf{Input:} initial point $\bx_0\in\BR^{1\times d}$, iteration number $T$, \\ stepsize $\eta>0$ and communication numbers $K$ \\[0.12cm]
\STATE $\mX^0=\vone \bx^0$ \\[0.12cm]
\STATE $\mS^0=\nabla \mF(\mX^0)$ \\[0.12cm]
\STATE \textbf{for} $t = 0, \dots, T-1$ \textbf{do}\\[0.12cm]
\STATE\quad $\mX^{t+1} = \FM(\mX^t - \eta \mS^t,\mW,K)$ \\[0.12cm]
\STATE\quad $\mS^{t+1} = \FM(\mS^{t} + \nabla\mF(\mX^{t+1}) - \nabla\mF(\mX^t),\mW,K)$ \\[0.12cm]
\STATE\textbf{end for} \\[0.12cm]
\STATE \textbf{Output:} uniformly sample $\vx^{\rm out}$ from $\{\vx^{T}(i)\}_{i=1}^n$   \\[0.12cm]
\end{algorithmic}
\end{algorithm}

\begin{algorithm}[t]
\caption{DRONE} \label{alg:DRONE}
\begin{algorithmic}[1]
\STATE \textbf{Input:} initial point $\bx^0\in\BR^{1\times d}$, mini-batch size $b$, iteration number $T$, stepsize $\eta>0$, probability $p,q\in(0,1]$ and communication numbers $K$. \\[0.12cm]
\STATE $\mX^0=\vone \bx^0$ \\[0.12cm]
\STATE $\mS^0=\mG^0=\nabla \mF(\mX^0)$ \\[0.12cm]
\STATE \textbf{for} $t = 0, \dots, T-1$ \textbf{do}\\[0.12cm]
\STATE \quad $\zeta^t \sim {\rm Bernoulli}(p)$ \\[0.12cm]
\STATE \quad $[\xi^t_1,\cdots,\xi^t_n]^\top \sim {\rm Multinomial}(b,q\vone)$ \\[0.12cm]
\STATE\quad $\mX^{t+1} = \FM(\mX^t - \eta \mS^t, \mW, K)$ \\[0.12cm]
\STATE\quad \textbf{parallel for} $i = 1, \dots, n$ \textbf{do}\\[0.12cm]
\STATE\quad\quad \textbf{if} $\zeta^t = 1$ \textbf{then} \\[0.12cm]
\STATE\quad\quad\quad $\vg^{t+1}(i) = \nabla f_{i} (\vx^{t+1}(i))$ \\[0.12cm]
\STATE\quad\quad \textbf{else}
\STATE\quad\quad\quad $\vg^{t+1}(i)=\vg^t(i) + \dfrac{\xi^t_i}{bq}\big(\nabla f_i(\vx^{t+1}(i)) - \nabla f_i(\vx^t(i))\big)$
\STATE\quad\quad \textbf{end if} \\[0.12cm]
\STATE\quad\textbf{end parallel for} \\[0.12cm]
\STATE\quad $\mS^{t+1} = \FM(\mS^{t}+\mG^{t+1}-\mG^{t},\mW,K)$ \\[0.12cm]
\STATE\textbf{end for} \\[0.12cm]
\STATE \textbf{Output:} uniformly sample $\vx^{\rm out}$ from $\{\vx^{T}(i)\}_{i=1}^n$   \\[0.12cm]
\end{algorithmic}
\end{algorithm}

It is well-known that GD achieves the linear convergence for minimizing the PL function \cite{karimi2016linear} on single machine. 
For decentralized optimization, it is natural to integrate GD with gradient tracking \cite{shi2015extra,nedic2009distributed,qu2017harnessing} and Chebyshev acceleration \cite{arioli2014chebyshev,scaman2017optimal,song2023optimal,ye2023multi}, leading to Algorithm~\ref{alg:DGD-GT} which is called decentralized gradient descent with gradient tracking (DGD-GT).
The following theorem shows that the communication complexity and the time complexity of DGD-GT nearly match the lower bounds shown in Corollary~\ref{thm:decentralized-lower}. 

\begin{theorem}\label{thm:DGD-GT}
We suppose that Assumption \ref{asm:lower}--\ref{asm:W} hold, then running Algorithm~\ref{alg:DGD-GT} (DGD-GT) with appropriate parameters setting achieves 
$\BE[f(\vx^{\rm out})-f^*] \leq \epsilon$ within 
communication complexity of 
$\tilde\fO\big(\kappa/\sqrt{\gamma}\log(1/\epsilon)\big)$, time complexity of $\tilde\fO\big(\kappa (1+\tau/\sqrt{\gamma}\,)\log(1/\epsilon)\big)$ 
and LFO complexity of $\fO\big(n\kappa\log(1/\epsilon)\big)$ 
in expectation. 
\end{theorem}

However, the upper bound on LFO complexity of DGD-GT (Algorithm \ref{alg:DGD-GT}) shown in Theorem~\ref{thm:DGD-GT} does not match the lower bound of $\Omega\big(n+\kappa\sqrt{n}\log(1/\epsilon)\big)$ provided by Corollary~\ref{cor:decentralized-lfo}.
Recall that the optimal IFO methods for non-distributed setting are based on the stochastic recursive `gradient~\cite{li2021page,zhou2019faster,wang2019spiderboost}.
We borrow this idea to construct the recursive gradient with respect to local agents, i.e., we update the local gradient estimator by
\begin{align*}
\begin{split}    
\vg^{t+1}(i)  
= \begin{cases}
 \nabla f_{i} (\vx^{t+1}(i)), & \zeta^t=1, \\[0.15cm]
\vg^t(i) + \dfrac{\xi^t_i}{bq}\big(\nabla f_i(\vx^{t+1}(i))-\nabla f_i(\vx^t(i))\big), & \zeta^t=0,
\end{cases}
\end{split}
\end{align*}
where we introduce random variables $\zeta^t \sim {\rm Bernoulli}(p)$ and $[\xi^t_1,\cdots,\xi^t_n]^\top \sim {\rm Multinomial}(b,q\vone)$ with some small probabilities $p$ and~$q$ which encourage only few of agents compute local gradients in most of iterations.
Similar to the procedure of DGD-GT (Algorithm \ref{alg:DGD-GT}), we can also introduce steps of gradient tracking and Chebyshev acceleration to improve the communication efficiency.
Finally, we achieve \underline{d}ecentralized \underline{r}ecursive l\underline{o}cal gradie\underline{n}t d\underline{e}scent (DRONE) method, which is formally presented in Algorithm~\ref{alg:DRONE}.

We analyze the complexity of DRONE by the following Lyapunov function
\begin{align*}
    \Phi^t=\mathbb{E}[f(\bx^t)-f^*]+\alpha U^t+\beta V^t+ L C^t,
\end{align*}
where $\alpha=2\eta/p$, $\beta=8L\rho^2n\eta^2$,
\begin{align*}
    &U^t=\mathbb{E}\Bigg\Vert\frac{1}{n}\sum_{i=1}^{n}\big(\vg^t(i)-\nabla f_i(\vx^t(i))\big)\Bigg\Vert^2, \\
    &V^t=\frac{1}{n}\mathbb{E}\Vert \mG^t-\nabla\mF(\mX^t)\Vert^2,  \\
    &\text{and} ~~~~ C^t= \BE\Vert\mX^t-\mathbf{1}\bx^t\Vert^2+\eta^2\mathbb{E}\Vert\mS^t-\mathbf{1}\bs^t\Vert^2.
\end{align*}      
Compared with the analysis of \citet{luo2022optimal,li2022destress,xin2022fast} for general nonconvex case, we additionally introduce the term of $f^*$ into our Lyapunov function to show the desired linear convergence under the PL condition, i.e., we can show that
\begin{align*}
    \Phi^T \leq (1-\mu\eta)^T\Phi^0
\end{align*}
by taking $\eta=\Theta(1/L)$. 
We present the convergence result of DRONE formally in the following theorem and corollary. 

\begin{theorem}\label{thm:com}
We suppose that Assumption \ref{asm:lower}--\ref{asm:W} hold, then running DRONE (Algorithm \ref{alg:DRONE}) with parameters setting 
\begin{align*}
\begin{split}    
&p\in\left[\frac{1}{n+1},\frac{1}{2}\right],~~~ b\in\left[\frac{1-p}{p},n\right],~~~\eta\leq\min\left\{\frac{1}{20L},\frac{p}{2\mu}\right\},   \\
&~q=\frac{1}{n},~~
K=\left\lceil\frac{\sqrt{2}\,(4+\log n)}{(\sqrt{2}-1)\sqrt{\gamma}}\right\rceil ~~ \text{and} ~~ T \geq \left\lceil\frac{1}{\mu\eta}\log\frac{\Phi^0}{\epsilon}\right\rceil
\end{split}
\end{align*}
achieves output satisfying $\BE[f(\vx^{\rm out})-f^*] \leq \epsilon$.
\end{theorem}

\begin{corollary}\label{cor:decentralized-upper}
Under the setting of Theorem \ref{thm:com}, running DRONE (Algorithm \ref{alg:DRONE}) by specifically taking
\begin{align*}
\begin{split}    
&p=\frac{1}{\min\{\sqrt{n},\kappa\}+1},~~ b=\left\lceil\frac{1-p}{p}\right\rceil,~~ \eta=\min\left\{\frac{1}{20L},\frac{p}{2\mu}\right\},   \\
&~q=\frac{1}{n},~~
K=\left\lceil\frac{\sqrt{2}\,(4+\log n)}{(\sqrt{2}-1)\sqrt{\gamma}}\right\rceil ~~ \text{and} ~~ T =\left\lceil\frac{1}{\mu\eta}\log\frac{\Phi^0}{\epsilon}\right\rceil
\end{split}
\end{align*}
achieves output satisfying $\BE[f(\vx^{\rm out})-f^*] \leq \epsilon$ within communication complexity of 
$\tilde\fO\big(\kappa/\sqrt{\gamma}\log(1/\epsilon)\big)$, time complexity of $\tilde\fO\big(\kappa (1+\tau/\sqrt{\gamma}\,)\log(1/\epsilon)\big)$ and LFO complexity of $\fO\big((n+\kappa\sqrt{n}\,)\log(1/\epsilon)\big)$ in expectation.
\end{corollary}

Note that the setting of $p=1/(\min\{\sqrt{n},\kappa\}+1)$ leads to $b=\fO(\min\{\sqrt{n},\kappa\})$, $\eta=\Theta\left({1}/{L}\right)$
and $T=\Theta(\kappa\log({1}/{\epsilon}))$, 
guarantees the algorithm nearly match the LFO lower bound (Corollary \ref{cor:decentralized-lfo}) in both cases of $n=\fO(\kappa^2)$ and $n=\Omega(\kappa^2)$.

As a comparison, the analysis of PAGE \cite{li2021page} (for single machine optimization) takes $p=\Theta(1/\sqrt{n}\,)$, which leads to that $b=\Theta(\sqrt{n}\,)$, $\eta=\Theta(\min\{1/L,1/(\mu\sqrt{n}\,)\})$ and $T=\Theta((\kappa+\sqrt{n}\,)\log(1/\epsilon))$.
If we directly apply these parameters to DRONE, it will result LFO complexity of $\fO\big((n+\kappa\sqrt{n}\,)\log(1/\epsilon)\big)$, 
communication complexity of 
$\tilde\fO\big((\kappa+\sqrt{n}\,)\log(1/\epsilon)/\sqrt{\gamma}\,\big)$ and time complexity of $\tilde\fO\big(({\kappa+\sqrt{n}\,)} (1+\tau/\sqrt{\gamma}\,)\log(1/\epsilon)\big)$ in expectation.
In the case of $n=\Omega(\kappa^2)$, such communication complexity and time complexity do not match the corresponding lower bounds (Theorem \ref{thm:decentralized-lower}). 
Intuitively, our analysis for DRONE considers the larger stepsize $\eta=\Theta(1/L)\geq\Theta(1/(\mu\sqrt{n}\,))$ than PAGE when $n=\Omega(\kappa^2)$, which is important to reduce the iteration numbers $T=\lceil(1/\mu\eta)\log(\Phi^0/\epsilon)\rceil$, also reduce the overall communication rounds $KT$ and the overall time cost $(1+K\tau)T$.

\section{Numerical Experiments}

We conduct numerical experiments to compare DRONE with  centralized gradient descent (CGD) and DGD-GT, where
CGD is a distributed extension of GD in client-server network. Please see Appendix \ref{appendix:experiments} for details.

We test the algorithms on the following three problems:
\begin{itemize}
    \item Hard instance: We follow the instance in the proof of Theorem \ref{thm:decentralized-lower} (Appendix \ref{appendix:instance-exp}) and specifically let 
\begin{align*}
   f_i(\vx)=\frac{16}{3}h_i^\fC\big(\sqrt{12a}\vx\big)
\end{align*}
for formulation (\ref{main:prob}). We set $T=2$, $t=72$, $\fC=\left\{1\right\}$ and $\sigma=29$ for $h_i^\fC:\BR^{Tt}\to\BR$.
\item Linear regression: We consider the problem 
\begin{align}\label{exp:obj}
\begin{split}    
\min_{\vx\in\BR^d} f(\vx)=\frac{1}{m}\sum_{j=1}^m \ell_j(\vx)
~\text{with}~ \ell_j(\vx) = (\va_j^\top\vx-b_j)^2
\end{split}
\end{align}
where $\va_j\in\BR^{d}$ is the feature vector of the $j$-th sample and $b_j\in\BR$ is its label. We allocate the $m$ individual loss on the $n$ agents, which leads to 
\begin{align}\label{exp:local}
   f_i(\vx)=\frac{n}{m}\sum_{j=\lceil m/n\rceil (i-1)+1}^{\min\{\lceil m/n\rceil i,m\}}(\va_j^\top\vx-b_j)^2.
\end{align}
We evaluate the algorithms on dataset ``DrivFace'' ($m=606$, $d=921,600$) \cite{diaz2016reduced} for this problem.
\item Logistic regression: The objective function and local functions of this problem are similar to the counterparts in formulation (\ref{exp:obj})--(\ref{exp:local}), but replace the loss function by
\begin{align*}
    \hat\ell_j(\vx) = \log\big( 1 + \exp(-b_{j}\va_{j}^\top \vx)\big).
\end{align*}
and require $b_j\in\{1,-1\}$.
We evaluate the algorithms on dataset ``RCV1'' ($m=20,242$, $d=47,236$) \cite{diaz2016reduced} for this problem.
\end{itemize}

For all above problems, we set $n=32$ and use linear graph for  network of DGD-GT and DRONE, leading to spectral gap  
$\gamma=\left(1-\cos({\pi}/{32})\right)/\left(1+\cos({\pi}/{32})\right)\approx 0.0024$.

We present empirical results for CGD, DGD-GT and DRONE on problems of hard instance, linear regression and logistic regression in Figure \ref{fig:lower bound}, \ref{fig:Linear regression} and \ref{fig:Logistic regression}, which includes the comparisons on LFO calls, communication rounds and running time.

We can observe that DRONE requires significantly less LFO calls, since we have showed only DRONE matches the lower complexity bound on LFO calls.
We also observed CGD needs much less communication rounds than DGD-GT and DRONE, which also leads to less running time.
This is because of the client-server framework in CGD does not suffers from the consensus error which cannot be avoided in  decentralized optimization.
It also validate our theoretical analysis that the linear graph heavily affect the convergence rate of decentralized algorithms.
Additionally, DGD-GT and DRONE have comparable communication rounds and running time for all of these problems, which also supports our theoretical results (see Table \ref{table:distributed}).

\begin{figure*}[t!]
	\centering
	\begin{minipage}{0.33\textwidth}
		\centering
		\includegraphics[width=\linewidth]{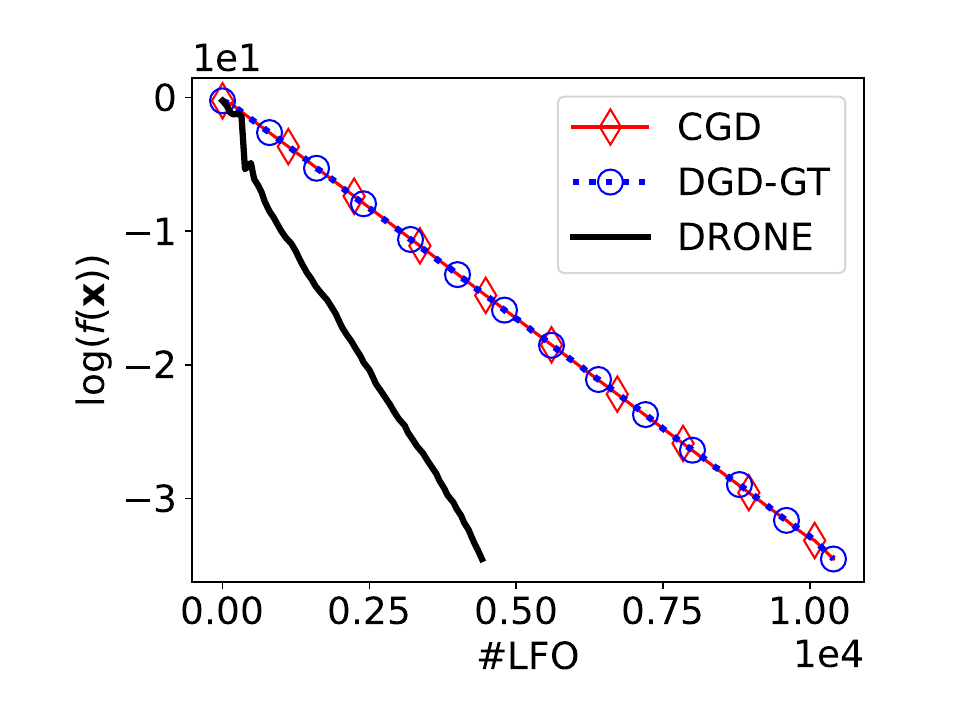} 
	\end{minipage}
	\begin{minipage}{0.33\textwidth}
		\centering
		\includegraphics[width=\linewidth]{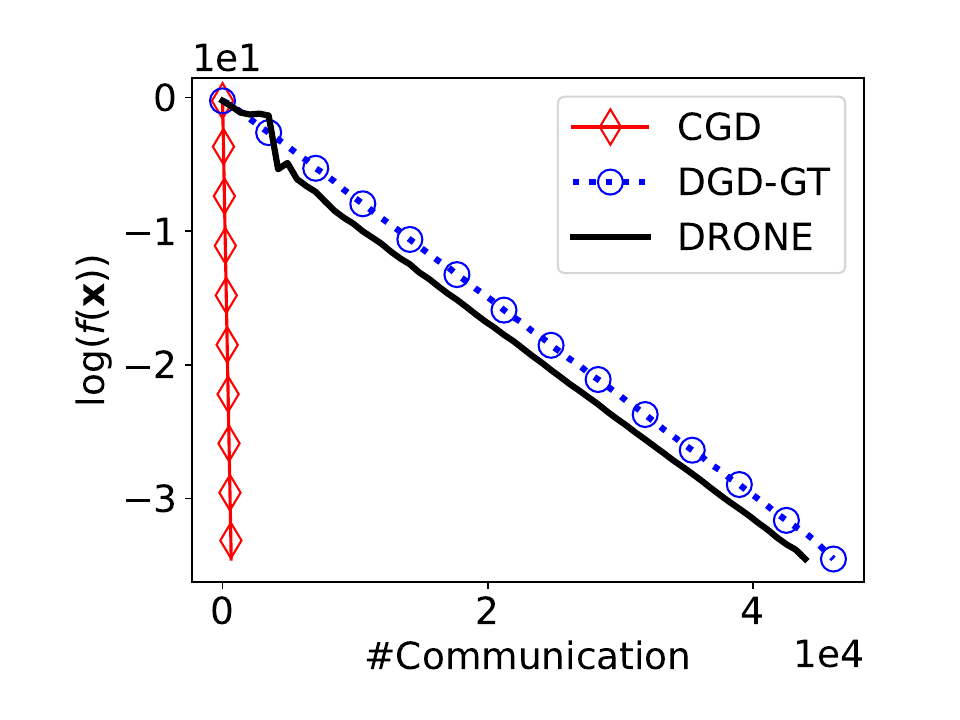} 
	\end{minipage}
	\begin{minipage}{0.33\textwidth}
		\centering
		\includegraphics[width=\linewidth]{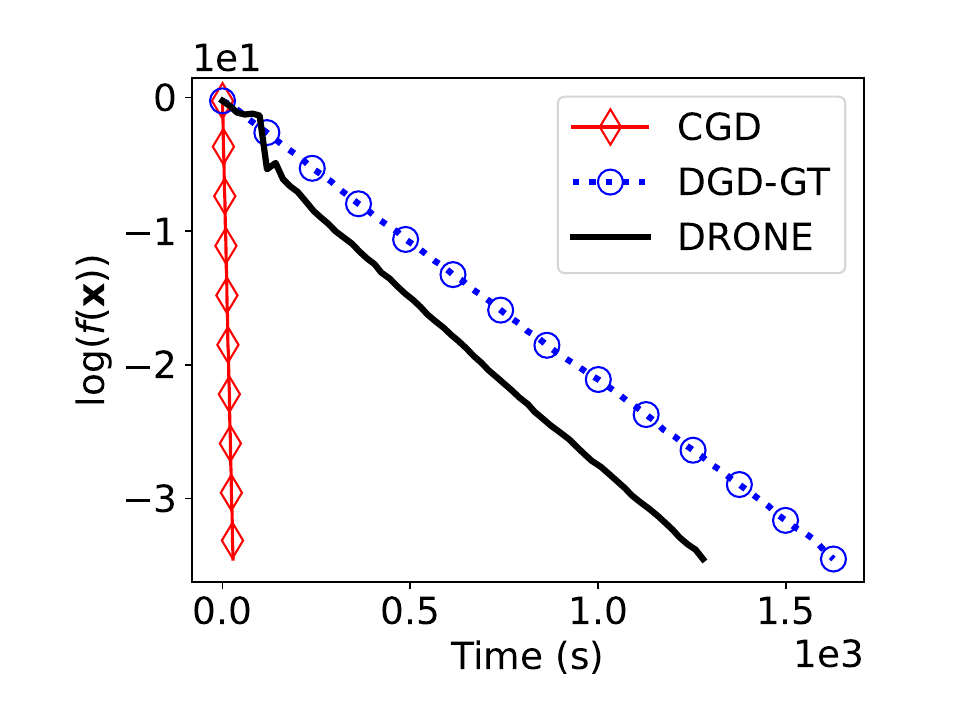}
	\end{minipage}
	\vspace{-0.4cm}
	\caption{The results for the hard instance in the proof of Theorem \ref{thm:decentralized-lower}.}
	\label{fig:lower bound} \vskip-0.25cm
\end{figure*}

\begin{figure*}[t!]
	\centering
	\begin{minipage}{0.33\textwidth}
		\centering
		\includegraphics[width=\linewidth]{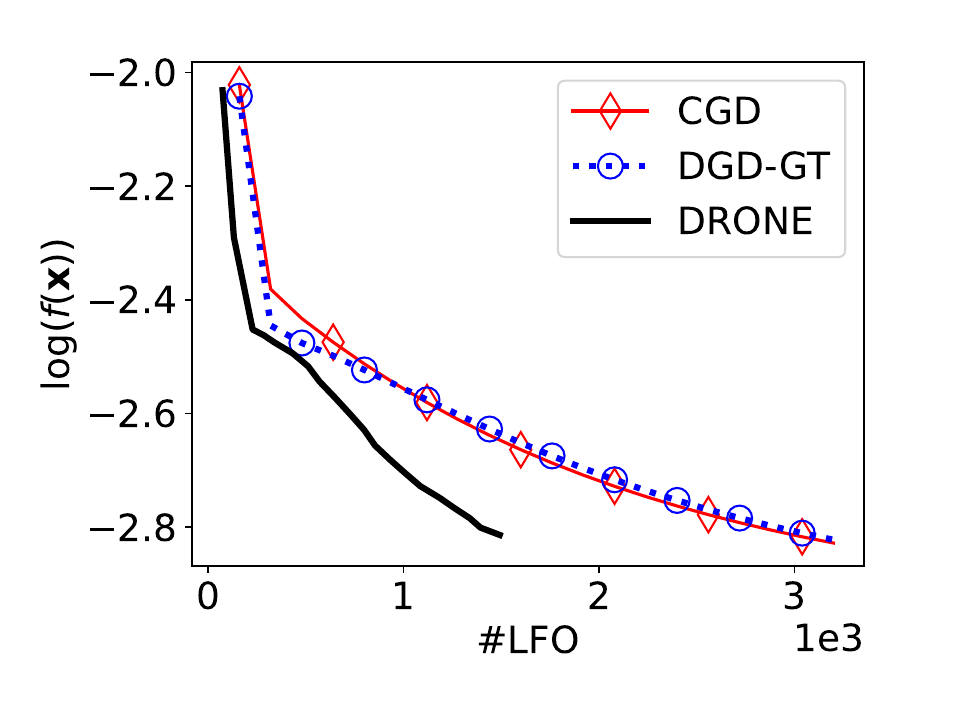}
	\end{minipage}
	\begin{minipage}{0.33\textwidth}
		\centering
		\includegraphics[width=\linewidth]{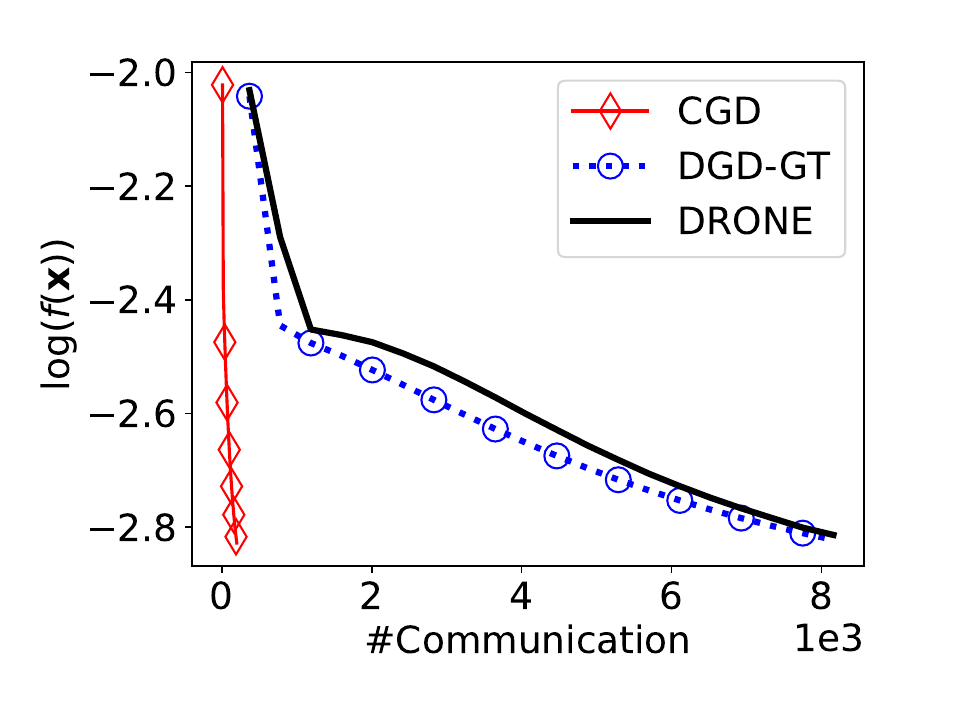} 
	\end{minipage}
	\begin{minipage}{0.33\textwidth}
		\centering
		\includegraphics[width=\linewidth]{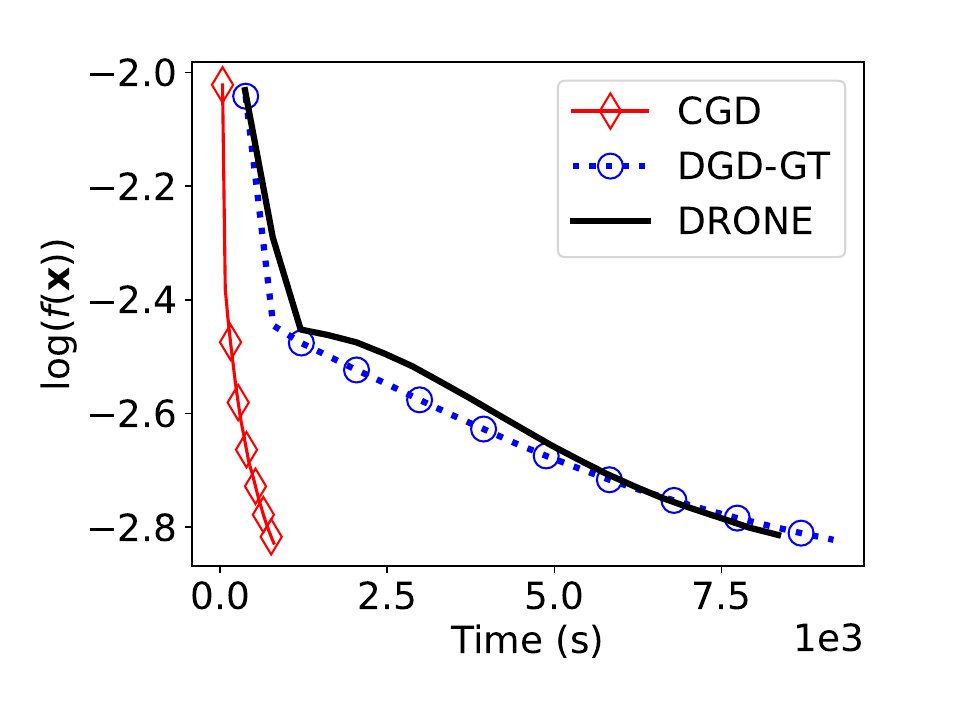}
	\end{minipage}
	\vspace{-0.4cm}
	\caption{The results for linear regression on dataset  ``DrivFace''.}
	\label{fig:Linear regression} \vskip-0.25cm
\end{figure*}

\begin{figure*}[t!]
	\centering
	\begin{minipage}{0.33\textwidth}
		\centering
		\includegraphics[width=\linewidth]{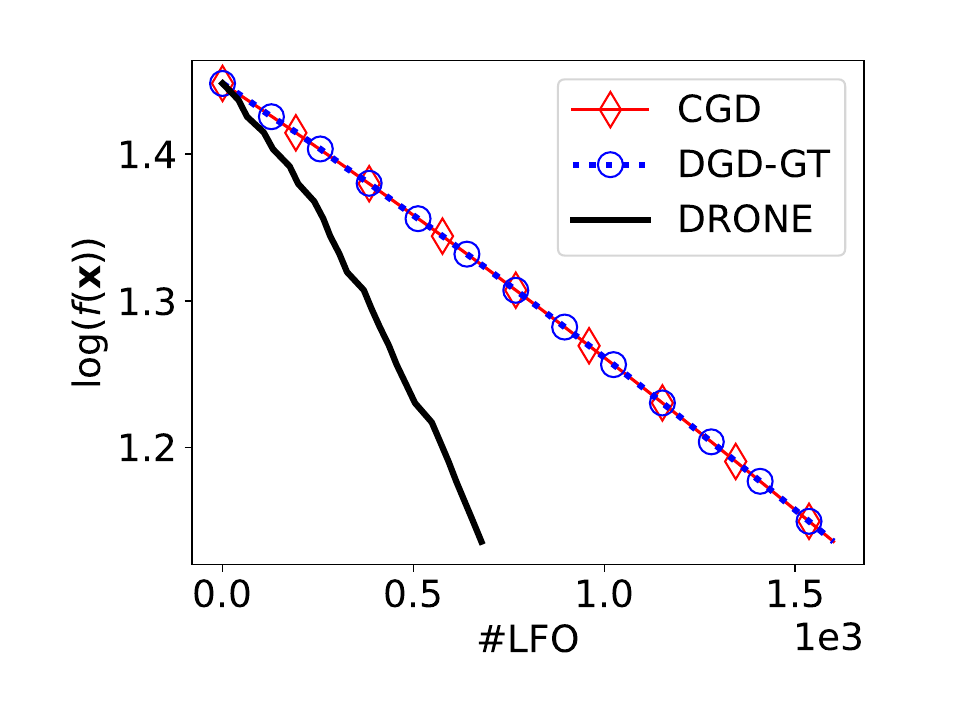}
	\end{minipage}
	\begin{minipage}{0.33\textwidth}
		\centering
		\includegraphics[width=\linewidth]{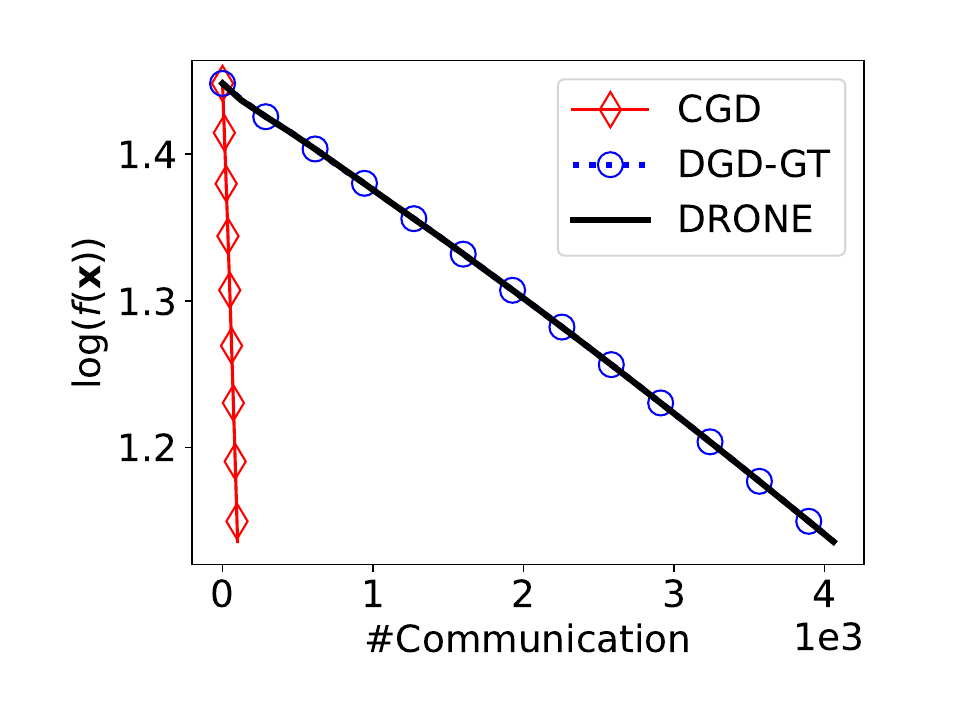}
	\end{minipage}
	\begin{minipage}{0.33\textwidth}
		\centering
		\includegraphics[width=\linewidth]{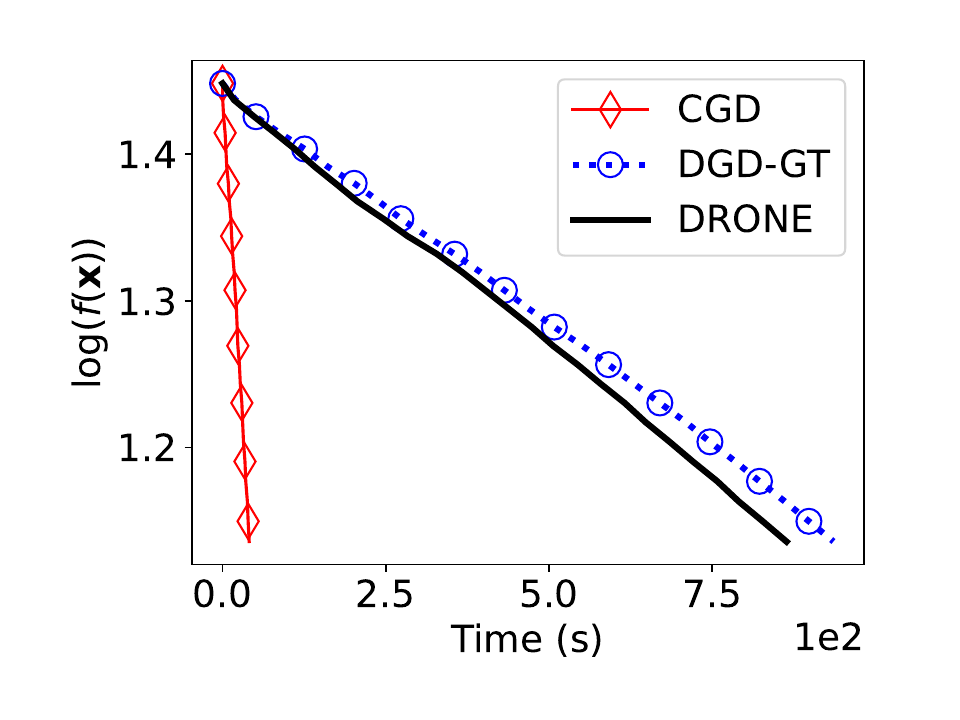}
	\end{minipage}
	\vspace{-0.4cm}
	\caption{The results for logistic regression on dataset  ``RCV1''.}
	\label{fig:Logistic regression}
\end{figure*}

\section{Conclusion}

We provide the lower complexity bound for smooth finite-sum optimization under the PL condition, 
which implies the upper bound of IFO complexity archived by existing first-order methods~\cite{wang2019spiderboost,li2021page,zhou2019faster} is nearly tight.
We also construct the lower bounds of communication complexity, time complexity and LFO complexity for minimizing the PL function in distributed setting and verify their tightness by proposing decentralized recursive local gradient descent.

In future work, we would like to study the lower bound in more general stochastic setting that the objective (or local functions) has the form of expectation~\cite{yuan2022revisiting,ACDFSW2019}. 
We are also interested in extending our results to address the functions that satisfy the Kurdyka--{\L}ojasiewicz inequality \cite{bolte2014proximal,bolte2007lojasiewicz,attouch2009convergence,zhou2018convergence,fatkhullin2022sharp,jiang2022holderian}.

\appendix

\section{The Proofs for Section \ref{sec:IFO}}\label{appendix:IFO}

Without the loss of generality, we always assume the IFO algorithm iterates with initial point $\vx^{(0)}=\vzero$. Otherwise, we can take functions $\{f_i(\vx+\vx^{(0)}\}_{i=1}^n$ into consideration.

Recall that we have defined the functions $\psi_\theta:\BR\to\BR$, $q_{T,t}:\BR^{Tt}\to\BR$ and $g_{T,t}:\BR^{Tt}\to\BR$ as  \cite{yue2023lower}
\begin{align}\label{eq:psi}
\psi_{\theta}(x)= 
\begin{cases}
\frac{1}{2}x^{2}, 
	& x\leq\frac{31}{32}\theta, \\[0.5em]
			\frac{1}{2}x^{2}-16(x-\frac{31}{32}\theta)^{2},& {\frac{31}{32}\theta<x\leq \theta,}\\[0.5em]
	\frac{1}{2}x^{2}-\frac{1}{32}\theta^2+16(x-\frac{33}{32}\theta)^{2}, &      {\theta<x\leq \frac{33}{32}\theta,}\\[0.5em]
	\frac{1}{2}x^{2}-\frac{1}{32}\theta^2,      
	& {x>\frac{33}{32}\theta,}
\end{cases} 
\end{align} 
\begin{align*}
		q_{T,t}(\vx)=\frac{1}{2}\sum_{i=0}^{t-1}\left(\left(\frac{7}{8}x_{iT}-x_{iT+1}\right)^2+\sum_{j=1}^{T-1}(x_{iT+j+1}-x_{iT+j})^2\right),
\end{align*}
and 
\begin{align}\label{eq:g}
		g_{T,t}(\vx)=q_{T,t}(\vb-\vx)+\sum_{i=1}^{Tt}\psi_{b_i}(b_i-x_i).
\end{align} 
where $x_0=0$ and $\vb\in \mathbb{R}^{Tt}$ with $b_{kT+\tau}=({7}/{8})^k$ for $k\in \{0\}\cup[t-1]$ and $\tau\in[T]$.
We can verify that 
\begin{align*}
g_{T,t}^*\triangleq \inf_{\vy\in\BR^{Tt}}g_{T,t}(\vy)=0.    
\end{align*}

\subsection{The Proof of Lemma \ref{lem:finite-sum}}
\begin{proof}
For any $\vx,\vy\in\BR^{mn}$, the smoothness of $g:\BR^m\to\BR$ implies
\begin{align*}
\Vert\nabla f_i(\vx)-\nabla f_i(\vy)\Vert 
= & \Big\Vert(\mU^{(i)})^\top\nabla g(U^{(i)}\vx)-(\mU^{(i)})^\top\nabla g(U^{(i)}\vy)\Big\Vert \\
= & \big\Vert\nabla g(\mU^{(i)}\vx)-\nabla g(\mU^{(i)}\vy)\big\Vert \\
\leq & \hat{L}\big\Vert \mU^{(i)}(\vx-\vy)\big\Vert \\
\leq & \hat{L}\Vert \vx-\vy\Vert,
\end{align*}
and
\begin{align*}
 \frac{1}{n}\sum_{i=1}^n\Vert\nabla f_i(\vx)-\nabla f_i(\vy)\Vert^2 
=& \frac{1}{n}\sum_{i=1}^{n}\Big\Vert(\mU^{(i)})^\top\nabla g(\mU^{(i)}\vx)-(\mU^{(i)})^\top\nabla g(\mU^{(i)}\vy)\Big\Vert^2 \\
=& \frac{1}{n}\sum_{i=1}^{n}\big\Vert\nabla g(\mU^{(i)}\vx)-\nabla g(\mU^{(i)}\vy)\big\Vert^2 \\
=& \frac{\hat{L}^2}{n}\sum_{i=1}^{n}\big\Vert \mU^{(i)}(\vx-\vy)\big\Vert^2 \\
\leq & \frac{\hat{L}^2}{n}\Vert \vx-\vy\Vert^2.
\end{align*}
This implies each $f_i:\BR^{mn}\to\BR$ is $\hat{L}$-smooth and $\{f_i:\BR^{mn}\to\BR\}_{i=1}^n$ is $\hat{L}/\sqrt{n}$-average smooth.

For any $\vx\in\BR^{mn}$, the PL condition of $g:\BR^m\to\BR$ implies
\begin{align*}
	\Vert\nabla f(\vx)\Vert^2& = \frac{1}{n^{2}}\Big\Vert\sum_{i=1}^{n}\nabla f_i(\vx)\Big\Vert^2 \\
	& = \frac{1}{n^{2}}\Big\Vert\sum_{i=1}^{n}(\mU^{(i)})^\top\nabla g(\mU^{(i)}\vx)\Big\Vert^2 \\
	& = \frac{1}{n^{2}}\sum_{i=1}^{n}\big\Vert\nabla g(\mU^{(i)}\vx)\big\Vert^2\\
	& \geq \frac{\hat{\mu}}{n^2}\sum_{i=1}^{n}\Big(g(\mU^{(i)}\vx)-\inf\limits_{\vx\in \mathbb{R}^{mn}}g(\mU^{(i)}\vx)\Big)\\
	& = \frac{\hat{\mu}}{n}(f(\vx)-f^*),
\end{align*}
which means $f:\BR^{mn}\to\BR$ is $\hat{\mu}/n$-PL.
  
Consider the facts $f(\vzero)=g(\vzero)$ and
\begin{align*}
      \inf\limits_{\vx\in \mathbb{R}^{mn}}\sum_{i=1}^{n}g(\mU^{(i)}\vx)=\sum_{i=1}^{n}\inf\limits_{\vx\in \mathbb{R}^{mn}}g(\mU^{(i)}\vx)=n\inf\limits_{\vx\in \mathbb{R}^{m}}g(\vx),
\end{align*}
then we have
\begin{align*}
     f(\mathbf{0})-\inf_{\vy\in\BR}f(\vy)=g(\mathbf{0})-\inf_{\vy\in\BR}g(\vy).
\end{align*}
\end{proof}

\subsection{The Proof of Lemma \ref{lem:scale}}

\begin{proof}
     For any $\vx,\vy\in\BR^{mn}$, the smoothness of $g:\BR^m\to\BR$ implies
     \begin{align*}
             \Norm{\nabla\hat g(\vx)-\nabla\hat g(\vy)}&=\alpha\beta\Norm{\nabla g(\beta \vx)-\nabla g(\beta \vy)}\\
             &\leq\alpha\beta \hat{L}\Norm{\beta \vx-\beta \vy}\\
             &\leq\alpha\beta^2 \hat{L}\Norm{\vx-\vy},
     \end{align*}
     which means $\hat g$ is $\alpha\beta^2 \hat{L}$-smooth.
     
     For any $\vx\in\BR^{mn}$, the PL condition of $g:\BR^m\to\BR$ implies
     \begin{align*}
             \Norm{\nabla\hat g(\vx)}^2
             =&\alpha^2\beta^2\Norm{\nabla g(\beta\vx)}^2\\
             \geq&2\alpha^2\beta^2\hat{\mu}(g(\beta\vx)-g^*)\\
             =&2\alpha\beta^2\hat{\mu}(\hat g(\vx)-\hat g^*),
     \end{align*}
     which means $\hat g$ is $\alpha\beta^2\hat{\mu}$-PL.

     We can verify that $\hat g(\mathbf{0})=\alpha g(\mathbf{0})$ and $\hat g^*=\alpha g^*$, which means
     \begin{align*}
         \hat g(\mathbf{0})-\hat g^*=\alpha (g(\mathbf{0})-g^*).
     \end{align*}
 \end{proof}

\subsection{The Proof of Theorem  \ref{thm:kn}}\label{sec:thm:kn}
 
\begin{proof}
We first take $g_{T,t}:\BR^{Tt}\to\BR$ by following equation (\ref{eq:g}) with 
\begin{align*}
T=\left\lfloor \frac{L}{37a\sqrt{n}\mu}\right\rfloor \qquad\text{and}\qquad t=2\left\lfloor\log_{{7}/{8}}\frac{3\epsilon}{\Delta}\right\rfloor.    
\end{align*}
The statements (b), (c) and~(d) of Lemma \ref{lem:basic function} means the function $g_{T,t}$ is $37$-smooth, $1/(aT)$-PL and satisfies
\begin{align*}
g_{T,t}(\vzero)-g_{T,t}^*\leq3T.    
\end{align*}
We apply Lemma \ref{lem:scale} with 
\begin{align*}
g(\vx)=g_{T,t}(\vx), \quad m=Tt, \quad \alpha=\frac{\Delta}{3T} \quad \text{and} \quad \beta=\sqrt{\frac{3\sqrt{n}LT}{37\Delta}},     
\end{align*}
which means the function $\hat g(\vx)=\alpha g_{T,t}(\beta\vx)$ is $37\alpha\beta^2$-smooth, $\alpha\beta^2/(aT)$-PL and satisfies 
\begin{align*}
\alpha g_{T,t}(\mathbf{0})-\alpha g_{T,t}^*\leq3\alpha T. 	    
\end{align*}
Then we apply Lemma \ref{lem:finite-sum} with 
\begin{align*}
g(\vx)=\hat g(\vx)=\alpha g_{T,t}(\beta\vx) \qquad\text{and}\qquad m=Tt,    
\end{align*}
which achieves $f_i(\vx)=\alpha g_{T,t}(\beta \mU^{(i)}\vx)$ and $f(\vx)=\frac{1}{n}\sum_{i=1}^{n}f_i(\vx)$ such that $\left\{f_i:\BR^{nTt}\to\BR\right\}_{i=1}^n$ is $37\alpha\beta^2/\sqrt{n}$-mean-squared smooth and $f:nTt$ is $\alpha\beta^2/(anT)$-PL with 
\begin{align*}
f(\vx^0)-f^*=\alpha g_{T,t}(\mathbf{0})-\alpha g_{T,t}^*\leq3\alpha T.    
\end{align*}
The choice of $\alpha=\Delta/(3T)$ and $\beta=\sqrt{3\sqrt{n}LT/(37\Delta)}$ and condition $L\geq37a\sqrt{n}\mu$ implies
\begin{align*}
	\frac{37\alpha\beta^2}{\sqrt{n}}=L, \qquad  \frac{\alpha\beta^2}{anT}\geq\mu \qquad\text{and}\qquad  3\alpha T=\Delta.
\end{align*}
Therefore, the function set 
$\left\{f_i\right\}_{i=1}^n$ is $L$-average smooth and the function $f$ is $\mu$-PL with $f(\vx^0)-f^*\leq\Delta$.

Let $\delta=2\epsilon/\Delta$, then we can write $t=2\lfloor\log_{{8}/{7}}{2}/{(3\delta)}\rfloor$.
Moreover, the assumption $\epsilon<0.005\Delta$ means
$\delta<0.01$.
Then the statement (e) of Lemma \ref{lem:basic function} and definition $f_i(\vx)=\alpha g_{T,t}(\beta \mU^{(i)}\vx)$ implies if $\vx\in\BR^{nTt}$ satisfies 
\begin{align*}
{\rm supp}(\mU^{(i)}\vx)\subseteq\left\{1,2,\cdots,Tt/2\right\},    
\end{align*}
then
\begin{align*}
   f_i(\vx)- \alpha g_{T,t}^* & = \alpha g_{T,t}(\beta U^{(i)}\vx)-\alpha g_{T,t}^*\\
&>3\alpha T\delta\\
&=2\epsilon.
\end{align*}

Now we show that any IFO algorithm require at least 
$\lfloor nTt/4\rfloor+1=\Omega(\kappa\sqrt{n}\log(1/\epsilon))$ IFO calls to achieve an $\epsilon$-suboptimal solution $\hat\vx$ of the problem.
We consider the vector $\vx\in\BR^{nTt}$ achieved by an IFO algorithm with at most $\lfloor nTt/4\rfloor$ IFO calls.
The zero-chain property of $g_{Tt}$ (statement (a) of Lemma \ref{lem:basic function}) means the vector $\vx$ has at most $\lfloor nTt/4 \rfloor$ non-zero entries.
We partition $\vx\in\BR^{nTt}$ into $n$ vectors $\vy^{(1)},\dots,\vy^{(n)}\in\BR^{Tt}$ such that $\vy^{(i)}=\mU^{(i)}\vx\in\BR^{Tt}$.
Then there at least $\lceil n/2\rceil$ vectors in $\{\vy^{(i)}\}_{i=1}^n$ such that each of them has at least $Tt/2$ zero entries.
The zero-chain property means
there exists index set $\fI\subseteq [n]$ with $|\fI|\geq \lceil n/2\rceil$ such that each~$i\in\fI$ satisfies $y^{(i)}_{Tt/2+1}=\dots=y^{(i)}_{Tt}=0$.
Therefor, the statement (e) of Lemma \ref{lem:basic function} implies
\begin{align*}
f_i(\vx)- \alpha g_{T,t}^* & > 2\epsilon,
\end{align*}
which leads to
\begin{align*}
\frac{1}{n}\sum_{i=1}^n f_i(\vx)- f^*
= &   \frac{1}{n}\sum_{i=1}^n f_i(\vx)- \alpha g_{T,t}^* \\
\geq &  \frac{1}{n}\sum_{i\in\fI} (f_i(\vx)- \alpha g_{T,t}^*) \\
>& \frac{1}{n}\cdot \lceil \frac{n}{2} \rceil \cdot 2\epsilon \\
\geq& \epsilon.
\end{align*}
Hence, finding an $\epsilon$-suboptimal solution of the problem requires at least $\lfloor nTt/4\rfloor+1=\Omega(\kappa\sqrt{n}\log(1/\epsilon))$ IFO calls.
\end{proof}

\subsection{The Proof of Theorem \ref{thm:n}}\label{sec:thm:n}
\begin{proof}
We prove this theorem by following \citet[Theorem 2]{li2021page}.
For any $i\in[n]$, we define $f_i:\BR^d\to\BR$ as 
\begin{align*}
    f_i(\vx)=c\left< \mathbf{u}_i, \vx \right>+\frac{L}{2}\Vert \vx\Vert^2,
\end{align*}
where 
\begin{align*}
c=\sqrt{L\Delta}, \qquad
d=2n^2, \qquad
\mathbf{u}_i=\Big[\BI\Big(\Big\lceil\frac{1}{2n}\Big\rceil=i\Big),\BI\Big(\Big\lceil\frac{2}{2n}\Big\rceil=i\Big)\cdots,\BI\Big(\Big\lceil\frac{2n^2}{2n}\Big\rceil=i\Big)\Big]^\top\in\mathbb{R}^{d}    
\end{align*}
and $\BI(\cdot)$ is the indicator function.
	
For any $\vx,\vy\in\BR^d$, we have
\begin{align*}
    \nabla f_i(\vx) = c\vu_i + L\vx
\end{align*}
for any $i\in[n]$, which implies
\begin{align*}
	\frac{1}{n}\sum_{i=1}^n\Vert\nabla f_i(\vx)-\nabla f_i(\vy)\Vert^2&=\frac{1}{n}\sum_{i=1}^n\Vert(c\mathbf{u}_i+L\vx)-(c\mathbf{u}_i+L\vy)\Vert^2\\
&=\frac{1}{n}\sum_{i=1}^n\Vert L(\vx-\vy)\Vert^2\\
&=L^2\Vert \vx-\vy\Vert^2.
\end{align*}
Hence, we conclude $\{f_i:\BR^{d}\to\BR\}_{i=1}^n$ is $L$-mean-squared smooth.

We also have $\nabla^2f(\vx)=L\mI\succeq\mu\mI$ for any $\vx\in\BR^d$. Hence, the function $\nabla^2f(\vx)$ is $\mu$-strongly convex, also is $\mu$-PL.
  
We have
\begin{align*}
f^* 
=& \frac{1}{n}\sum_{i=1}^{n}\left(c\inner{\mathbf{u}_i}{\vx^*} +\frac{L}{2}\Vert \vx^*\Vert^2\right)\\
=& \frac{c}{n}\sum_{i=1}^{n} \inner{\mathbf{u}_i}{\vx^*}  +\frac{L}{2}\Vert \vx^*\Vert^2\\
=&-\frac{c^2}{2Ln^2}\bigg\Vert\sum_{i=1}^{n}\mathbf{u}_i\bigg\Vert^2\\
=& -\frac{c^2}{L},
\end{align*}
where $\vx^*=-({c}/{Ln})\vone$ is the minima of $f:\BR^d\to\BR$.
Then the optimal function value gap holds 
\begin{align*}
f(\vx^0)-f^* = 0 - f^* =\frac{c^2}{L} = \Delta.
\end{align*}

We consider any IFO algorithm with initial point $\vx^0=\vzero$.
After $t$ IFO calls, Definition \ref{def:linear-span} implies
\begin{align*}
    \vx^t \in {\rm Lin}\big(\{\nabla f_{i_0}(\vx^0),\dots,\nabla f_{i_{t-1}}(\vx^{t-1})\}\big) = {\rm Lin}\big(\{\vu_{i_0},\dots,\vu_{i_{t-1}}\}\big),
\end{align*}
where $i_\tau\in[n]$ is the index of individual which is accessed at the $\tau$-th IFO calls.
Since each $\vu_{i_\tau}$ has $2n$ nonzero entries, any vector $\vx\in\BR^d$ achieved by at most $n/2$ IFO calls has at least 
\begin{align*}
    d - \frac{n}{2}\cdot 2n = 2n^2 - n^2 = n^2
\end{align*}
zero entries.
Let $\fI_0=\{j\in[2n^2]:x_j=0\}$, then we have $|\fI|\geq n^2$.
Based on the construction of $f_i$ and $\vu_i$, we have
\begin{align*}
	f(\vx)-f^*
= & \frac{1}{n}\sum_{i=1}^{n} \left(c\inner{\vu_i}{\vx} + \frac{L}{2}\norm{\vx}\right) - \left(-\frac{c^2}{L}\right) \\
= & \sum_{j=1}^{2n^2}\left(\frac{c}{n}x_j+\frac{L}{2}x_j^2+\frac{c^2}{2Ln^2}\right) \\
= & \sum_{j\in\fI_0}\left(\frac{c}{n}x_j+\frac{L}{2}x_j^2+\frac{c^2}{2Ln^2}\right) + \sum_{j\not\in\fI_0}\left(\frac{c}{n}x_j+\frac{L}{2}x_j^2+\frac{c^2}{2Ln^2}\right) \\
\geq & n^2\cdot\frac{c^2}{2Ln^2} + \sum_{j\not\in\fI_0}\left(x_j+\frac{c}{nL}\right)^2 \\
\geq & \frac{\Delta}{2}>\epsilon,
\end{align*}
Hence, achieving an $\epsilon$-suboptimal solution requires at least $n/2+1=\Omega(n)$ IFO calls.
\end{proof}

\subsection{The Proof of Corollary \ref{cor:ifo}}
\begin{proof}
This result can be achieved by directly combining Theorem \ref{thm:kn} and \ref{thm:n}.    
\end{proof}

\section{The Proofs for Section \ref{sec:decentralized}}\label{appendix:decentralized}
Without loss of generality, we always assume that all agents start with the internal memory of null space, i.e., we have~$\mathcal{M}_i^0=\left\{\mathbf{0}\right\}$ for any $i\in[n]$.

The main idea in our lower bound analysis splitting the function $g_{T,t}:\BR^{Tt}\to\BR$ defined in equation (\ref{eq:g}) by introducing the functions $q_1:\BR^{Tt}\to\BR$, $q_2:\BR^{Tt}\to\BR$ and $r:\BR^{Tt}\to\BR$ as 
\begin{align}
&	q_1(\vx)=\frac{1}{2}\sum_{i=1}^{Tt/2}(x_{2i-1}-x_{2i})^2, \label{eq:q1} \\
&	q_2(\vx)=\frac{1}{2}\sum_{i=0}^{t-1}\left[\Big(\frac{7}{8}x_{iT}-x_{iT+1}\Big)^2+\sum_{j=iT/2+1}^{(i+1)T/2-1} (x_{2j}-x_{2j+1})^2\right], \label{eq:q2} \\
& r(\vx)=\sum_{i=1}^{Tt}\psi_{b_i}(b_i-x_i), \label{eq:r}
\end{align}
where we suppose $T$ is even and let $x_0=0$.
Then we can verify that the function $g_{T,t}(\cdot)$ can be written as
\begin{align*}
    g_{T,t}(\vx)=q_1(\vb-\vx)+q_2(\vb-\vx)+r(\vx).
\end{align*}

We let $\fG=\{\fV,\fE\}$ be the graph associated to the network  of the agents in decentralized optimization, where the node set~$\fV=\{1,\dots,n\}$ corresponds to the $n$ agents and the edge set $\fE=\{(i,j):\text{node $i$ and node $j$ are connected}\}$ describes the topology of the agents network.  

For given a subset $\fC\subseteq\mathcal{V}$, we define the function $h_i^\fC(\vx):\BR^{Tt}\rightarrow\BR$ as
\begin{align}\label{eq:h}
h_i^\fC(\vx)=
\begin{cases}
\frac{r(\vx)}{n}+\frac{q_1(\vb-\vx)}{\vert \fC\vert}   &{i\in \fC}, \\[0.2em]
\frac{r(\vx)}{n}+\frac{q_2(\vb-\vx)}{\vert \fC_{\sigma}\vert}   &{i\in \fC_{\sigma}}, \\[0.2em]
\frac{r(\vx)}{n},   &{\text{otherwise}},\\    
\end{cases}
\end{align}
where 
$\fC_{\sigma}=\left\{v\in\mathcal{V}:{\rm dis}(\fC,v)\geq \sigma\right\}$ and ${\rm dis}(\fC,v)$ is the distance between set $\fC$ and node $v$.

We introduce the following property \citep[Lemma 4]{yue2023lower} of function $\psi_\theta:\BR\to\BR$ to analyze the smoothness of $h_i^\fC(\cdot)$ and the communication complexity for hard instance.

\begin{lemma}\label{lem:psi}
    For any $\theta>0$, the function $\psi_\theta:\BR\to\BR$ defined in equation (\ref{eq:psi}) is 33-smooth and holds $\psi_{\theta}'(\theta)=0$.
\end{lemma}

Now we present the proofs for lower bounds in decentralized setting, which is based on our construction (\ref{eq:q1})--(\ref{eq:r}).

\subsection{Proof of Lemma \ref{lem:decentralized instance}}
\begin{proof} 
We can observed that the function $q_1(\vx)$ is quadratic and holds $q_1(\vx)\leq 1$ for any $\vx\in\BR^{Tt}$, then it is $2$-smooth. 
Similarly, the function $q_2(\vx)$ is also $2$-smooth.
		
For any $\vx,\vy\in\BR^{Tt}$, we have
\begin{equation*}
			\begin{split}
				\Norm{\nabla r(\vx)-\nabla r(\vy)}^2&=\sum_{i=1}^{Tt}\left(v'_{b_i}(b_i-x_i)-v'_{b_i}(b_i-y_i)\right)^2\\
				&\leq\sum_{i=1}^{Tt}\left(33(x_i-y_i)\right)^2\\
				&=33^2\Norm{\vx-\vy}^2,\\
			\end{split}
		\end{equation*}
where inequality is based on the Lemma \ref{lem:psi}. This implies the function $r:\BR^{Tt}\to\BR$ is 33-smooth.
	
Combing above smoothness properties and the definition of $h_i^\fC(\cdot)$, we conclude each $h_i^\fC$ is $(33/n+\max\left\{2/\vert \fC\vert,2/\vert \fC_\sigma\vert\right\})$-smooth and thus $\big\{h_i^\fC:\BR^{Tt}\to\BR\big\}_{i=1}^n$ is $(33/n+\max\left\{2/\vert \fC\vert,2/\vert \fC_\sigma\vert\right\})$-mean-squared smooth.
	
The definition of $h(\cdot)$ and $h_i^\fC(\cdot)$ implies
\begin{align}\label{eq:hg}
    h(\vx) = \frac{1}{n}\sum_{i=1}^n h_i^\fC(\vx) = \frac{q_1(\vb-\vx)+q_2(\vb-\vx)+r(\vx)}{n} = \frac{g_{T,t}(\vx)}{n}.
\end{align}
Then applying statements (c) and (d) of Lemma \ref{lem:basic function} and Lemma \ref{lem:scale} finish the proof for the last two statements.
\end{proof}

\subsection{Proof of Lemma \ref{lem:communication time}}
\begin{proof}
The definitions $f_i(\vx)=\alpha h_i^\fC(\beta \vx)$ and $f(\vx)=\frac{1}{n}\sum_{i=1}^{n}f_i(\vx)$ and equation (\ref{eq:hg}) implies
$f(\vx)=\alpha g_{T,t}(\beta\vx)/n$.
Then the statement (e) of Lemma \ref{lem:basic function} implies when $\text{supp}(\vx)\subseteq\left\{1,2,\cdots,Tt/2\right\}$, we have
\begin{align}
	f(\vx)-f^*=\frac{\alpha}{n}(g_{T,t}(\beta\vx)-g_{T,t}^*)>\frac{3\alpha T\delta}{n},
\end{align}
since we have assumed  $\delta<0.01$ and $t=2\lfloor\log_{{8}/{7}}({2}/{3\delta})\rfloor$.

We define
\begin{align*}
    {\rm nnz}(s,i) \triangleq
     \max\left\{m\in\BN:\text{there exists $\vy\in\mathcal{M}_i^s$ such that $y_m\neq0$}\right\}.
\end{align*}
Then statements (a) of Lemma \ref{lem:basic function} implies achieving an $3\alpha T\delta/n$-suboptimal solution requires 
\begin{align}\label{ieq:nnz}
    \mathop{\max}_{i\in[n]}\left\{{\rm nnz}(s,i)\right\}\geq \frac{Tt}{2}+1.
\end{align}
Now we consider how much local computation steps and local communication steps we need to attain the condition (\ref{ieq:nnz}). 

According to Lemma \ref{lem:psi} and equation (\ref{eq:r}), for any $\vx\in\mathcal{M}_i^s$, we have
\begin{equation}\label{eq:others}
	(\nabla r(\vx))_k=0,\qquad \text{for any}\quad k> {\rm nnz}(s,i).
\end{equation}
According to equation (\ref{eq:q1}), for any $\vx\in\mathcal{M}_i^s$, we have
\begin{equation}\label{eq:C}
	(\nabla q_1(\vb-\vx))_k=0, \qquad \text{for any}\quad k> {\rm nnz}(s,i)+\BI\left\{{\rm nnz}(s,i)\equiv 1 \pmod{2}\right\}.
\end{equation}
According to equation (\ref{eq:q2}), for any $\vx\in\mathcal{M}_i^s$, we have
\begin{equation}\label{eq:C_sigma}
	(\nabla q_2(\vb-\vx))_k=0, \qquad \text{for any}\quad k> {\rm nnz}(s,i)+\BI\left\{{\rm nnz}(s,i)\equiv 0 \pmod{2}\right\}.
\end{equation}
Combining (\ref{eq:h}), (\ref{eq:others}), (\ref{eq:C}), (\ref{eq:C_sigma}) and Definition \ref{dfn:DFO}, we know that for any DFO algorithm:
\begin{enumerate}
    \item If $i\in\fC$ and ${\rm nnz}(s,i)$ is odd, one step of local computation can increase at most one dimension for memory of node $i$. 
    \item If $i\in\fC_\sigma$ and ${\rm nnz}(s,i)$ is even, one step of local computation can increase at most one dimension for memory of node~$i$.
    \item Otherwise, one step of local computation cannot increase the dimension for memory of node~$i$.
\end{enumerate}
In summary, we have
\begin{align}\label{eq:nnz}
    {\rm nnz}(s+1,i) \leq \begin{cases}
        {\rm nnz}(s,i)+1,  & \text{if}~~i\in\fC,\  {\rm nnz}(s,i)\equiv 1 \pmod{2}, \\
        {\rm nnz}(s,i)+1, & \text{if}~~i\in\fC_\sigma,\  {\rm nnz}(s,i)\equiv 0 \pmod{2}, \\
        {\rm nnz}(s,i), &\text{otherwise}.
    \end{cases}
\end{align}

We consider the cost to reach the second coordinate from the initial status that $\fM_i^0=\{\vzero\}$ for all $i\in[n]$.
According to equation (\ref{eq:nnz}), we need to let a node in $\fC$ reach the first coordinate, which requires at least one local computation step on some node in $\fC_\sigma$ first. 
Then, according to definitions of DFO algorithm (Definition \ref{dfn:DFO}) and $\fC_\sigma$, one must perform at least $\sigma$ local communication steps for a node in $\fC$ to receive the information of the first coordinate from some node in $\fC_\sigma$. 
After above steps, we can perform at least 1 computation on nodes in $\fC$ to reach the second coordinate.
In summary, to reach the second coordinate requires at least 2 local computation step and $\sigma$ local communication step.

Similarly, to reach the $k$-th coordinate, a DFO algorithm must perform at least $k$ local computation steps and $(k-1)\sigma$ local communication steps.
Thus, to attain the condition (\ref{ieq:nnz}), one needs at least $Tt/2+1$ local computation steps and $Tt\sigma/2$ local communication steps, which corresponds to $Tt\sigma/2$ communications round and $Tt(1+\sigma\tau)/2$ time cost.
\end{proof}
\begin{remark}
Noticing that one computation step corresponds to one unit of time cost. However, some of agents (maybe not all agents) can parallel compute their local gradient, which means the computational time cost may be not proportion to the number of local gradient oracle calls.    
\end{remark}

\subsection{Proof of Theorem \ref{thm:decentralized-lower}}\label{appendix:instance-exp}
\begin{proof}
We consider the instance graph provided by \citet{scaman2017optimal}, which associated to the specific spectral gap.
Concretely, we let $\iota_m=\left(1-\cos({\pi}/{m})\right)/\left(1+\cos({\pi}/{m})\right)$.
For given $\gamma\in(0,1]$, let
$m=\big\lfloor{\pi}/{\arccos{{((1-\gamma)}/{(1+\gamma))}}}\big\rfloor$,
then we have $m\geq2$ and $\iota_{m+1}<\gamma\leq\iota_m$.
We study the cases of $m\geq3$ and $m=2$ separately.
		
We first consider the case of $m\geq3$. 
Let the agent number $n=m$, take $\mathcal{G}=\{\fV,\fE\}$ be the undirected linear graph of size~$n$ ordered from node $1$ to node $n$ such that $\fV=\{1,\dots,n\}$ and $\fE=\{(i,j):|i-j|=1, i\in\fV ~~\text{and}~~ j\in\fV\}$. 
We define a weighted matrix $\hat\mW_l\in\BR^{n\times n}$ for $\fG$ such that $\hat w_{i+1,i}=\hat w_{i,i+1}=1-l\BI(i=1)$ and $\hat w_{ij}=0$ for other entries.
Let~$\hat\mR_l\in\BR^{n\times n}$ be the Laplacian matrix of graph $\mathcal{G}$ associated to weighted matrix $\hat\mW_l$. We define $\lambda_1(\hat\mR_l),\dots,\lambda_n(\hat\mR_l)$ be eigenvalues of $\hat\mR_l$ such that $0=\lambda_n(\hat\mR_l)\leq\lambda_{n-1}(\hat\mR_l)\leq\cdots\leq\lambda_1(\hat\mR_l)$. 
A simple calculation gives that 
\begin{align*}
\lambda_{n-1}(\hat\mR_0)=2\left(1-\cos({\pi}/{m})\right), \quad
\lambda_{1}(\hat\mR_0)=2\left(1+\cos({\pi}/{m})\right), \quad
\lambda_{n-1}(\hat\mR_1)=0 \qquad\text{and} \qquad
\lambda_{1}(\hat\mR_l)>0     
\end{align*}
for any $l\in[0,1]$, then we have $\lambda_{n-1}(\hat\mR_0)/\lambda_{1}(\hat\mR_0)=\iota_m$ and $\lambda_{n-1}(\hat\mR_1)/\lambda_{1}(\hat\mR_1)=0$.
By the continuity of the eigenvalues of a matrix and the fact $0<\gamma\leq\iota_m$, there exists some $l\in\left[0,1\right)$ such that $\lambda_{n-1}(\hat\mR_l)/\lambda_{1}(\hat\mR_l)=\gamma$. 
Let $\mW=1-\hat\mR_l/\lambda_1(\hat\mR_l)$, then the spectral gap satisfies $\gamma(\mW)=\lambda_{n-1}(\hat\mR_l)/\lambda_{1}(\hat\mR_l)=\gamma$. According to basic properties of Laplacian matrix, the matrix~$\mW\in\BR^{n\times n}$ is a mixing matrix satisfies Assumption \ref{asm:W}. 

We take $\left\{h_i^\fC:\BR^{Tt}\to\BR\right\}_{i=1}^n$ by following equation (\ref{eq:h}) with 
\begin{align*}
\fC=\left\{1,\dots, \left\lceil \frac{n}{32}\right\rceil\right\},\qquad
\sigma=\left\lceil \frac{15n}{16}\right\rceil-1,\qquad T=2\left\lfloor\frac{\kappa}{194a}\right\rfloor 
\qquad\text{and}\qquad 
t=2\left\lfloor \log_{\frac{8}{7}}\frac{2\Delta}{3\epsilon}\right\rfloor.  
\end{align*}
Let $h(\cdot)=\frac{1}{n}\sum_{i=1}^{n}h_i^\fC(\cdot)$.
Lemma \ref{lem:decentralized instance} means the function set $\left\{h_i^\fC:\BR^{Tt}\to\BR\right\}_{i=1}^n$ is $(33/n+\max\left\{2/\vert \fC\vert,2/\vert \fC_{\sigma}\vert\right\})$-mean-squared smooth, and the function $h$ is $1/\left(anT\right)$-PL and satisfies
\begin{align*}
    h(\mathbf{0})-h^* \leq \frac{3T}{n}.
\end{align*}
Let $f_i(\vx)=\alpha h_i^\fC(\beta \vx)$, $f(\vx)=\sum_{i=1}^{n}f_i(\vx)/n$ and apply Lemma \ref{lem:scale} with $\alpha=n\Delta/(3T)$ and $\beta=\sqrt{3LT/(97\Delta)}$, then the function set $\left\{f_i:\BR^{Tt}\to\BR\right\}_{i=1}^n$ is $\alpha\beta^2(33/n+\max\left\{2/\vert \fC\vert,2/\vert \fC_{\sigma}\vert\right\})$-mean-squared smooth, and the function $f:\BR^{Tt}\to\BR$ is $\alpha\beta^2/\left(anT\right)$-PL and satisfies
\begin{align*}
    f(\mathbf{0})-f^*\leq \frac{3\alpha T}{n}.
\end{align*}
The values of $\alpha$ and $\beta$ means
\begin{align*}
& \alpha\beta^2\left(\frac{33}{n}+\frac{2}{\vert \fC_\sigma\vert}\right)
\leq\alpha\beta^2\left(\frac{33}{n}+\frac{2}{\vert \fC\vert}\right)
\leq\frac{nL}{97}\left(\frac{33}{n}+\frac{64}{n}\right)
=L, \\
& \frac{\alpha\beta^2}{anT}=\frac{L}{97aT}\geq\mu
\qquad \text{and} \qquad
\frac{3\alpha T}{n}=\Delta.
\end{align*}
According to Lemma \ref{lem:scale} and Lemma \ref{lem:decentralized instance}, we conclude $\left\{f_i:\BR^{Tt}\to\BR\right\}_{i=1}^n$ is $L$-mean-squared smooth and $f:\BR^d\to\BR$ is $\mu$-PL with $f(\mathbf{0})-f^*\leq\Delta$.
	
Applying Lemma \ref{lem:communication time} with $\delta=\epsilon/\Delta<0.01$, any DFO algorithm needs at least $Tt\sigma/2$ communication steps and $Tt(1+\sigma\tau)/2$ time cost to achieve an $\epsilon$-suboptimal solution since $3\alpha T\delta/n=\epsilon$.
The setting $n=m$ implies
\begin{align*}
\frac{2}{(n+1)^2}\leq\iota_{n+1}<\gamma\leq\iota_3=\frac{1}{3},
\end{align*}
which means
\begin{align*}
	\sigma=\left\lceil \frac{15n}{16} \right\rceil-1\geq\frac{15}{16}\left(\sqrt{\frac{2}{\gamma}}-1\right)-1\geq\frac{1}{5\sqrt{\gamma}}.
\end{align*}
Hence, we achieve the lower bounds for communication complexity and time complexity of $Tt\tau/2=\Omega\big(\kappa/\sqrt{\gamma}
\log(1/\epsilon)\big)$ and $Tt(1+\sigma\tau)/2=\Omega\big(\kappa(1+\tau/\sqrt{\gamma})\log(1/\epsilon)\big)$ respectively.  

We then consider the case of $m=2$. 
We let the agent number be $n=3$ and take $\mathcal{G}=\{\fV,\fE\}$ be the totally connected graph of size $n$ such that $\fV=\{1,\dots,n\}$ and $\fE=\{(i,j):i\in\fV ~~\text{and}~~ j\in\fV\}$.
We define a weighted matrix $\hat\mW_l\in\BR^{n\times n}$ for $\fG$ such that $\hat w_{1,3}=\hat w_{3,1}=l$ and $w_{ij}=1$ for other entries.
Let $\hat\mR_l\in\BR^{n\times n}$ be the Laplacian matrix of graph $\mathcal{G}$ associated to weighted matrix $\hat\mW_l$. We define $\lambda_1(\hat\mR_l),\dots,\lambda_n(\hat\mR_l)$ be eigenvalues of $\hat\mR_l$ such that~$0=\lambda_n(\hat\mR_l)\leq\lambda_{n-1}(\hat\mR_l)\leq\cdots\leq\lambda_1(\hat\mR_l)$. 
A simple calculation gives that 
\begin{align*}
\lambda_{n-1}(\hat\mR_1)=\lambda_{1}(\hat\mR_1)=3, \quad
\lambda_{n-1}(\hat\mR_0)=2\left(1-\cos({\pi}/{n})\right), \quad
\lambda_{1}(\hat\mR_0)=2\left(1+\cos({\pi}/{n})\right) \quad\text{and} \quad
\lambda_{1}(\hat\mR_l)>0     
\end{align*}
for any $l\in[0,1]$.
Then we have $\lambda_{n-1}(\hat\mR_0)/\lambda_{1}(\hat\mR_0)=\iota_n=\iota_3$ and $\lambda_{n-1}(\hat\mR_1)/\lambda_{1}(\hat\mR_1)=1$.
By continuity of the eigenvalues of a matrix and the fact $\iota_3<\gamma\leq1$, there exists some $l\in\left(0,1\right]$ such that $\lambda_{n-1}(\hat\mR_l)/\lambda_{1}(\hat\mR_l)=\gamma$.
Let~$\mW=1-\hat\mR_l/\lambda_1(\hat\mR_l)$, then the spectral gap satisfies $\gamma(\mW)=\lambda_{n-1}(\hat\mR_l)/\lambda_{1}(\hat\mR_l)=\gamma$. 
According to basic properties of Laplacian matrix, the matrix~$\mW\in\BR^{n\times n}$ is a mixing matrix satisfies Assumption \ref{asm:W}. 
		
We take $\left\{h_i^\fC:\BR^{Tt}\to\BR\right\}_{i=1}^n$ by following equation (\ref{eq:h}) with 
\begin{align*}
\fC=\left\{1\right\},\qquad\sigma=1,\qquad T=2\left\lfloor\frac{\kappa}{78a}\right\rfloor \qquad\text{and}\qquad t=2\left\lfloor \log_{\frac{8}{7}}\frac{2\Delta}{3\epsilon}\right\rfloor.    
\end{align*}
Let $h(\cdot)=\frac{1}{n}\sum_{i=1}^{n}h_i^\fC(\cdot)$.
Lemma \ref{lem:decentralized instance} means the function set $\left\{h_i^\fC:\BR^{Tt}\to\BR\right\}_{i=1}^n$ is $(33/n+\max\left\{2/\vert \fC\vert,2/\vert \fC_{\sigma}\vert\right\})$-mean-squared smooth, $h$ is $1/\left(anT\right)$-PL and satisfies
\begin{align*}
    h(\mathbf{0})-h^*\leq3T/n.
\end{align*}
Let $f_i(\vx)=\alpha h_i^\fC(\beta \vx)$, $f(\vx)=\sum_{i=1}^{n}f_i(\vx)/n$ and apply Lemma \ref{lem:scale} with $\alpha=n\Delta/(3T)$ and $\beta=\sqrt{LT/(13\Delta)}$, then we conclude $\left\{f_i:\BR^d\to\BR\right\}_{i=1}^n$ is $\alpha\beta^2(33/n+\max\left\{2/\vert \fC\vert,2/\vert \fC_{\sigma}\vert\right\})$-mean-squared smooth and $f$ is $\alpha\beta^2/\left(anT\right)$-PL and satisfies
\begin{align*}
    f(\mathbf{0})-f^* \leq \frac{3\alpha T}{n}.
\end{align*}

The values of $\alpha$ and $\beta$ means
\begin{align*}
	 \alpha\beta^2\left(\frac{33}{n}+\frac{2}{\vert \fC_\sigma\vert}\right)
			<\alpha\beta^2\left(\frac{33}{n}+\frac{2}{\vert \fC\vert}\right)
	=\frac{nL}{39}\left(\frac{33}{n}+\frac{6}{n}\right)
	=L,  \qquad
	 \frac{\alpha\beta^2}{anT}=\frac{L}{39aT}\geq\mu \qquad\text{and}\qquad
	\frac{3\alpha T}{n}=\Delta.
\end{align*}
According to Lemma \ref{lem:scale} and Lemma \ref{lem:decentralized instance}, we conclude $\left\{f_i:\BR^{Tt}\to\BR\right\}_{i=1}^n$ is $L$-mean-squared smooth, and $f:\BR^d\to\BR$ is $\mu$-PL with $f(\mathbf{0})-f^*\leq\Delta$.
		
Applying Lemma \ref{lem:communication time} with $\delta=\epsilon/\Delta<0.01$, any DFO algorithm needs at least $Tt\sigma/2$ communications and $Tt(1+\sigma\tau)/2$ time to reach an 
$\epsilon$-suboptimal solution since $3\alpha T\delta/n=\epsilon$.
The setting $n=3$ means $\gamma>\iota_3=1/3$ and $\sigma=1>1/\sqrt{3\gamma}$, which results the lower bounds for communication complexity and time complexity of $Tt\tau/2=\Omega\big(\kappa/\sqrt{\gamma}
\log(1/\epsilon)\big)$ and~$Tt(1+\sigma\tau)/2=\Omega\big(\kappa(1+\tau/\sqrt{\gamma})\log(1/\epsilon)\big)$ respectively.

\end{proof}
 
\section{The Proofs in Section \ref{sec:algorithm}}\label{appendix:convergence}

Recall that we define the Lyapunov function
\begin{align*}
    \Phi^t=\mathbb{E}[f(\bx^t)-f^*]+\alpha U^t+\beta V^t+LC^t,
\end{align*}
where $\alpha=2\eta/p$, $\beta=8L\rho^2n\eta^2$,
\begin{align*}
    U^t=\mathbb{E}\Big\Vert\frac{1}{n}\sum_{i=1}^{n}\big(\vg^t(i)-\nabla f_i(\vx^t(i))\big)\Big\Vert^2,~~
    V^t=\frac{1}{n}\mathbb{E}\Vert \mG^t-\nabla\mF(\mX^t)\Vert^2 
    ~~\text{and}~~
    C^t=\BE\Vert\mX^t-\mathbf{1}\bx^t\Vert^2+\eta^2\mathbb{E}\Vert\mS^t-\mathbf{1}\bs^t\Vert^2.
\end{align*}  
Compared with the analysis of \citet{luo2022optimal,li2022destress} for general nonconvex case, we introduce the term of $f^*$ into the Lyapunov function to show the linear convergence under the PL condition.
Different with previous work \cite{luo2022optimal,li2022destress} suppose each $f_i(\cdot)$ is $L$-smooth, our analysis only require $\{f_i(\cdot)\}_{i=1}^n$ is $L$-smooth that allows each~$f_i(\cdot)$ to be $\sqrt{n}L$-smooth (see Lemma \ref{lem:smooth}).

The remainder of this section first provide some technical Lemmas, then give the detailed proofs for the results in Section~\ref{sec:decentralized}.

\subsection{Some Technical Lemmas}

We introduce some lemmas for our later analysis.
\begin{lemma}[{\citet[Proposition 1]{ye2023multi}}]\label{lem:FM} 
Under Assumption \ref{asm:W}, Algorithm~\ref{alg:acc-gossip} holds 
\begin{align*}
    \frac{1}{n}\vone^\top\mY^K = \bar\vy^0
    \qquad\text{and}\qquad
	\big\|\mY^K-\mathbf{1}\bar\vy^0\big\| \leq  \sqrt{14}\left(1-\left(1-\frac{1}{\sqrt{2}}\right)\sqrt{1-\lambda_2(\mW)}\right)^K \big\|\mY^0-\mathbf{1}\bar\vy^0\big\|,
\end{align*}
where $\bar\vy^{0} = \frac{1}{n}\mathbf{1}^\top\mY^0=\frac{1}{n}\sum_{i=1}^n \mY^0(i)$.
\end{lemma} 

\begin{lemma}[{\citep[Lemma 3]{ye2023multi}}]\label{lem:avg-norm}
For any $\mX\in\BR^{n\times d}$, we have~$\Norm{\mX - \vone\bx} \leq \Norm{\mX}$ where $\mX=\frac{1}{n}\vone^\top\mX$.
\end{lemma}

\begin{lemma}\label{lem:smooth}
Under Assumption \ref{asm:as}, the function $f(\cdot)=\frac{1}{n}\sum_{i=1}^n f_i(\cdot)$ is $L$-smooth and each $f_i(\cdot)$ is $\sqrt{n}L$-smooth.
\end{lemma}
\begin{proof}
For any $\vx,\vy\in\BR^d$, the mean-square smoothness of $\{f_i:\BR^d\to\BR\}_{i=1}^n$ implies
\begin{align*}
    \norm{\nabla f(\vx)- \nabla f(\vy)}^2 
= & \norm{\frac{1}{n}\sum_{i=1}^n (\nabla f_i(\vx)- \nabla f_i(\vy))}^2 \\
\leq & \frac{1}{n}\sum_{i=1}^n \norm{\nabla f_i(\vx)- \nabla f_i(\vy)}^2 \\
\leq & L\norm{\vx- \vy}^2, 
\end{align*}
which means $f(\cdot)$ is $L$-smooth.

For any $\vx,\vy\in\BR^d$ and $i\in[n]$, the mean-square smoothness of $\{f_i:\BR^d\to\BR\}_{i=1}^n$ implies
\begin{align*} 
        \Vert\nabla f_i(\vx)-\nabla f_i(\vy)\Vert\leq\sqrt{\sum_{i=1}^n\Norm{\nabla f_i(\vx)-\nabla f_i(\vy)}^2}\leq\sqrt{nL^2\Norm{\vx-\vy}^2}\leq\sqrt{n}L\Norm{\vx-\vy},
\end{align*}
which means $f_i(\cdot)$ is $\sqrt{n}L$-smooth.
\end{proof}

\begin{lemma}[{\citet[Lemma 2]{li2021page}}]\label{lem:decrease}
Suppose the function $f:\BR^d\to\BR$ is $L$-smooth and the vectors $\vx_t,\vx_{t+1},\vv_t\in\BR^d$ satisfy $\vx_{t+1}=\vx_{t} - \eta\vv_t$ for $\eta>0$. Then we have
\begin{align}
f(\vx_{t+1})
\leq f(\vx_t) - \frac{\eta}{2}\Norm{\nabla f(\vx_t)}^2 - \left(\frac{1}{2\eta}-\frac{L}{2}\right)\Norm{\vx_{t+1}-\vx_t}^2 + \frac{\eta}{2}\Norm{\vv_t-\nabla f(\vx_t)}^2.
\end{align}
\end{lemma}

We establish the decrease of function value as follows.

\begin{lemma}\label{lem:linear} Under Assumption~\ref{asm:lower}-\ref{asm:W}, Algorithm~\ref{alg:DRONE} holds that
\begin{align*}
        \mathbb{E}[f(\bx^{t+1})-f^*]\leq(1-\mu\eta)\mathbb{E}[f(\bx^{t})-f^*]+\eta U^t+L^2\eta C^t-\left(\frac{1}{2\eta}-\frac{L}{2}\right)\mathbb{E}\Vert\bx^{t+1}-\bx^{t}\Vert^2.
\end{align*}
\begin{proof}
Lemma \ref{lem:FM} and the update $\mX^{t+1} = \FM(\mX^t - \eta \mS^t, \mW, K)$ means
\begin{align}\label{update:avg}
\begin{split}
\bx^{t+1} = & \frac{1}{n}\vone^\top\FM(\mX^t-\eta\mS^t,\mW,K) \\
= & \frac{1}{n}\vone^\top(\mX^t-\eta\mS^t) \\
= & \bx^t - \eta \bs^t.
\end{split}
\end{align} 
Lemma \ref{lem:smooth} shows the function $f(\cdot)$ is $L$-smooth, then 
Lemma \ref{lem:decrease} with $\vx^{t}=\bx^{t}$, $\vx^{t+1}=\bx^{t+1}$ and $\vv^{t}=\bs^{t}$ means
\begin{align}\label{ieq:decrease-f1}
\begin{split}
f(\bx^{t+1})
\leq f(\bx^t) - \frac{\eta}{2}\Norm{\nabla f(\bx^t)}^2 - \left(\frac{1}{2\eta}-\frac{L}{2}\right)\Norm{\bx^{t+1}-\bx^t}^2 + \frac{\eta}{2}\Norm{\bs^t-\nabla f(\bx^t)}^2.
\end{split}
\end{align}
We also have
\begin{align*}
\BE\Norm{\frac{1}{n}\sum_{i=1}^n\big(\nabla f_i(\vx^t(i))-\nabla f_i(\bx^t)\big)}^2 
\leq & \frac{1}{n}\sum_{i=1}^n\BE\Norm{\nabla f_i(\vx^t(i))-\nabla f_i(\bx^t)}^2 \\
\leq & \frac{1}{n}\sum_{i=1}^m nL^2\BE\Norm{\vx^t(i)-\bx^t}^2 \\
= & L^2\BE\Norm{\vx^t-\vone\bx^t}^2 
\leq L^2C^t,
\end{align*}
where we use the inequality $\Norm{\frac{1}{n}\sum_{i=1}^n \va_i}^2 \leq \frac{1}{n} \sum_{i=1}^n \Norm{\va_i}^2$ for $\va_1,\dots,\va_n\in\BR^d$ and Lemma~\ref{lem:smooth}.
Consequently, we have
\begin{align}\label{ieq:error-grad}
\begin{split}
  \BE\Norm{\bs^t-\nabla f(\bx^t)}^2 
=& \BE\Norm{\frac{1}{n}\sum_{i=1}^n \big(\vg^t(i)-\nabla f_i(\bx^t)\big)}^2 \\
\leq &  2\BE\Norm{\frac{1}{n}\sum_{i=1}^n\big(\vg^t(i)-\nabla f_i(\vx^t(i))\big)}^2 +  2\BE\Norm{\frac{1}{n}\sum_{i=1}^n\big(\nabla f_i(\vx^t(i))-\nabla f_i(\bx^t)\big)}^2 \\
\leq & 2U^t +  2L^2C^t.
\end{split}
\end{align}
Combining the results of (\ref{ieq:decrease-f1}) and (\ref{ieq:error-grad}), we have
\begin{align*} 
 \BE[f(\bx^{t+1})-f^*]
\leq & \BE\left[f(\bx^t)-f^* - \frac{\eta}{2}\Norm{\nabla f(\bx^t)}^2  + \eta U^t + L^2 \eta C^t - \left(\frac{1}{2\eta}-\frac{L}{2}\right)\Norm{\bx^{t+1}-\bx^t}^2\right] \\
\leq & (1-\mu\eta)\mathbb{E}[f(\bx^{t})-f^*]+\eta U^t+L^2\eta C^t-\left(\frac{1}{2\eta}-\frac{L}{2}\right)\mathbb{E}\Vert\bx^{t+1}-\bx^{t}\Vert^2,
\end{align*}
where the last step is based on the PL condition in Assumption~\ref{asm:SC}.
\end{proof}
\end{lemma}

Then we provide the recursion for $U_t$, $V_t$ and $C_t$ in the following lemma.

\begin{lemma}\label{lem:CUV}
Under the setting of Theorem \ref{thm:com}, we have
\begin{align*}
& C^{t+1}\leq 20\rho^2n^2C^t+4\rho^2np\eta^2V^t+12\rho^2n^2\mathbb{E}\Vert\bx^{t+1}-\bx^{t}\Vert^2, \\
& U^{t+1}\leq(1-p)U^t+4pL^2C^t+3pL^2\BE\Norm{\bx^{t+1} - \bx^t}^2, 
\end{align*}       
and
\begin{align*}
 V^{t+1}\leq(1-p)V^t+4pnL^2C^t+3pnL^2\BE\Norm{\bx^{t+1} - \bx^t}^2,
\end{align*}       
where $\rho=\sqrt{14}\left(1-\big(1-{1}/{\sqrt{2}}\big)\sqrt{1-\lambda_2}\,\right)^K$.
\end{lemma}

\begin{proof}
The setting Lemma \ref{lem:FM} implies
\begin{align}\label{ieq:rho}
\rho^2 \leq \frac{1}{80n^2}\leq \frac{1}{80}.    
\end{align}
We first consider 
$C^{t+1}=\BE\Vert\mX^{t+1}-\mathbf{1}\bx^{t+1}\Vert^2+\eta^2\mathbb{E}\Vert\mS^{t+1}-\mathbf{1}\bs^{t+1}\Vert^2$. 
For the term $\BE\Vert\mX^{t+1}-\mathbf{1}\bx^{t+1}\Vert^2$, we have
\begin{align}\label{ieq:CX}
\begin{split}    
    \mathbb{E}\Vert\mX^{t+1}-\mathbf{1}\bx^{t+1}\Vert^2&\leq\rho^2\mathbb{E}\Vert(\mX^t-\eta \mS^t)-\mathbf{1}(\bx^t-\eta\bs^t)\Vert^2\\
&\leq2\rho^2\big(\mathbb{E}\Vert\mX^t-\mathbf{1}\bx^t\Vert^2+\eta^2\mathbb{E}\Vert\mS^t-\mathbf{1}\bs^t\Vert^2\big)\\
&=2\rho^2C^t,
\end{split}
\end{align}
where we use the definition of $\rho$ and Lemma \ref{lem:FM} in the first inequality and Young’s inequality in the second inequality.

For the term $\eta^2\Vert\mS^t-\mathbf{1}\bs^t\Vert^2$, we have
\begin{align}\label{ieq:CS}
\begin{split}
    \eta^2\mathbb{E}\Vert\mS^{t+1}-\mathbf{1}\bs^{t+1}\Vert^2 
\leq & \rho^2\eta^2\mathbb{E}\Vert\mS^t+\mG^{t+1}-\mG^t-\mathbf{1}(\bs^t+\bg^{t+1}-\bg^{t})\Vert^2\\
\leq &2\rho^2\eta^2\mathbb{E}\Vert\mS^{t}-\mathbf{1}\bs^{t}\Vert^2+2\rho^2\eta^2\mathbb{E}\Vert\mG^{t+1}-\mG^t-\mathbf{1}(\bg^{t+1}-\bg^{t})\Vert^2\\
\leq & 2\rho^2\eta^2\mathbb{E}\Vert\mS^{t}-\mathbf{1}\bs^{t}\Vert^2+2\rho^2\eta^2\mathbb{E}\Vert\mG^{t+1}-\mG^t\Vert^2\\
\leq & 2\rho^2C^t+2\rho^2\eta^2\mathbb{E}\Vert\mG^{t+1}-\mG^t\Vert^2,\\
\end{split}
\end{align}
where we use the definition of $\rho$ and Lemma \ref{lem:FM} in the first inequality, Young’s inequality in the second inequality, Lemma \ref{lem:avg-norm} in the third inequality and definition of $C^t$ in the last inequality.

We bound the term $\Vert\mG^{t+1}-\mG^{t}\Vert^2$ in inequality (\ref{ieq:CS}) as
\begin{align}\label{ieq:CG}
\small\begin{split}
  & \mathbb{E}\Vert\mG^{t+1}-\mG^{t}\Vert^2 \\
=& p\mathbb{E}\Vert\nabla \mF(\mX^{t+1})-\mG^{t}\Vert^2+(1-p)\sum_{i=1}^n\mathbb{E}\Big\Vert\dfrac{\xi_i^t}{bq}\big(\nabla f_i(\vx^{t+1}(i)) - \nabla f_i(\vx^t(i))\big)\Big\Vert^2\\
\leq & 2p\mathbb{E}\Vert\nabla\mF(\mX^{t+1})-\nabla\mF(\mX^{t})\Vert^2+2p\mathbb{E}\Vert\nabla\mF(\mX^{t})-\mG^t\Vert^2+(1-p)\sum_{i=1}^n\mathbb{E}\Big\Vert\dfrac{\xi_i^t}{bq}\big(\nabla f_i(\vx^{t+1}(i)) - \nabla f_i(\vx^t(i))\big)\Big\Vert^2\\
\leq & 2p\mathbb{E}\Vert\nabla\mF(\mX^{t+1})-\nabla\mF(\mX^{t})\Vert^2+2p\mathbb{E}\Vert\nabla\mF(\mX^{t})-\mG^t\Vert^2+\frac{(1-p)(b+n)}{b}\sum_{i=1}^n\mathbb{E}\Vert\nabla f_i(\vx^{t+1}(i)) - \nabla f_i(\vx^t(i))\Vert^2\\
=& 2pnV^t+\Big(2p+\frac{(1-p)(b+n)}{b}\Big)\mathbb{E}\Vert\nabla\mF(\mX^{t+1})-\nabla\mF(\mX^{t})\Vert^2\\
\leq & 2pnV^t+\big(2p+p(b+n)\big)\mathbb{E}\Vert\nabla\mF(\mX^{t+1})-\nabla\mF(\mX^{t})\Vert^2\\
\leq & 2pnV^t+4pn\mathbb{E}\Vert\nabla\mF(\mX^{t+1})-\nabla\mF(\mX^{t})\Vert^2
\end{split}
\end{align}
where the first inequality is based on Cauchy--Schwarz inequality, the second inequality is based on
\begin{align*}
    \mathbb{E}\Big\Vert\dfrac{\xi_i^t}{bq}\big(\nabla f_i(\vx^{t+1}(i)) - \nabla f_i(\vx^t(i))\big)\Big\Vert^2 \leq \frac{\mathbb{E}[(\xi_i^t)^2]}{b^2q^2}\mathbb{E}\Vert\nabla f_i(\vx^{t+1}(i)) - \nabla f_i(\vx^t(i))\Vert^2
\end{align*}
and the fact $\xi_i^t\sim {\rm Binomial}(b,q)$ that leads to
\begin{align*}
   \frac{\mathbb{E}[(\xi_i^t)^2]}{b^2q^2}
=\frac{b^2q^2+bq(1-q)}{b^2q^2}
=\frac{b+n-1}{b}
\leq\frac{b+n}{b},
\end{align*}
the third inequality is based on the fact $b\geq(1-p)/p$ and the last step is based on the fact that $1\leq b\leq n$.

We bound the term $\mathbb{E}\Vert\nabla\mF(\mX^{t+1})-\nabla\mF(\mX^{t})\Vert^2$ in inequality (\ref{ieq:CG}) as
\begin{align}\label{ieq:CF}
\begin{split}
& \mathbb{E}\Vert\nabla\mF(\mX^{t+1})-\nabla\mF(\mX^{t})\Vert^2 \\
=& \sum_{i=1}^n\BE\Norm{\nabla f_{i}(\vx^{t+1}(i)) - \nabla f_{i}(\vx^t(i))}^2\\
\leq & 3\sum_{i=1}^n\BE\big(\Norm{\nabla f_{i}(\vx^{t+1}(i)) - \nabla f_{i}(\bx^{t+1})}^2 +\Norm{\nabla f_{i}(\bx^{t+1}) - \nabla f_{i}(\bx^t)}^2+\Norm{\nabla f_{i}(\bx^t) - \nabla f_{i}(\vx^t(i))}^2\big)\\
\leq & 3(\sqrt{n}L)^2\sum_{i=1}^n \BE\Norm{\vx^{t+1}(i) - \bx^{t+1}}^2 + 3nL^2\BE\Norm{\bx^{t+1} - \bx^t}^2 + 3(\sqrt{n}L)^2\sum_{i=1}^n \BE\Norm{\vx^{t}(i) - \bx^{t}}^2\\
= & 3nL^2\big(\BE\Norm{\mX^{t+1} - \vone\bx^{t+1}}^2 + \BE\Norm{\mX^t - \vone\bx^t}^2+\BE\Norm{\bx^{t+1} - \bx^t}^2\big) \\
\leq & 3nL^2\big(2\rho^2C^t + C^t+\BE\Norm{\bx^{t+1} - \bx^t}^2\big) \\
\leq & 4nL^2C^t+3nL^2\BE\Norm{\bx^{t+1} - \bx^t}^2,\\
\end{split}
\end{align}
where the the second inequality is based on Lemma \ref{lem:smooth}, the third inequality is based on inequality (\ref{ieq:CX}) and last step is based on the the fact $\rho^2\leq1/6$.

Combining inequalities (\ref{ieq:CX}) -- (\ref{ieq:CF}), we achieve
\begin{align*}
C^{t+1} \leq 20\rho^2n^2C^t+4\rho^2np\eta^2V^t+12\rho^2n^2\mathbb{E}\Vert\bx^{t+1}-\bx^{t}\Vert^2.
\end{align*}

We then consider $U^{t+1}=\mathbb{E}\big\Vert\frac{1}{n}\sum_{i=1}^{n}\big(\vg^{t+1}(i)-\nabla f_i(\vx^{t+1}(i))\big)\big\Vert^2$. 
We let $\xi$ be the random variable satisfying $P(\xi=i)=q$.
The update rule for $\vg_{t+1}(i)$ implies
\begin{align*}
U^{t+1}&=(1-p)\mathbb{E}\Big\Vert\frac{1}{n}\sum_{i=1}^{n}\Big(\vg^t(i)+\dfrac{\xi_i^t}{bq}\big(\nabla f_i(\vx^{t+1}(i)) - \nabla f_i(\vx^t(i))\big)-\nabla f_i(\vx^{t+1}(i))\Big)\Big\Vert^2\\
&=(1-p)U^t+(1-p)\mathbb{E}\Big\Vert\frac{1}{n}\sum_{i=1}^{n}\Big(\dfrac{\xi_i^t}{bq}\big(\nabla f_i(\vx^{t+1}(i)) - \nabla f_i(\vx^t(i))\big)-\big(\nabla f_i(\vx^{t+1}(i)))-\nabla f_i(\vx^{t}(i))\big)\Big)\Big\Vert^2\\
&=(1-p)U^t+(1-p)\mathbb{E}\Big\Vert\sum_{i=1}^{n}\Big(\dfrac{\xi_i^t}{b}\big(\nabla f_i(\vx^{t+1}(i)) - \nabla f_i(\vx^t(i))\big)-\frac{1}{n}\big(\nabla f_i(\vx^{t+1}(i)))-\nabla f_i(\vx^{t}(i))\big)\Big)\Big\Vert^2\\
&=(1-p)U^t+\frac{1-p}{b}\mathbb{E}\Big\Vert\big(\nabla f_{\xi}(\vx^{t+1}(\xi)) - \nabla f_{\xi}(\vx^t(\xi))\big)-\frac{1}{n}\sum_{i=1}^{n}\big(\nabla f_i(\vx^{t+1}(i)) - \nabla f_i(\vx^t(i))\big)\Big\Vert^2\\
&\leq(1-p)U^t+\frac{1-p}{b}\mathbb{E}\Big\Vert\big(\nabla f_{\xi}(\vx^{t+1}(\xi)) - \nabla f_{\xi}(\vx^t(\xi))\big)\Big\Vert^2\\
&=(1-p)U^t+\frac{1-p}{nb}\sum_{i=1}^n\mathbb{E}\Norm{\nabla f_{i}(\vx^{t+1}(i)) - \nabla f_{i}(\vx^t(i))}^2 \\
&\leq(1-p)U^t+\frac{p}{n}\sum_{i=1}^n\mathbb{E}\Norm{\nabla f_{i}(\vx^{t+1}(i)) - \nabla f_{i}(\vx^t(i))}^2 \\
&\leq(1-p)U^t+4pL^2C^t+3pL^2\BE\Norm{\bx^{t+1} - \bx^t}^2, 
\end{align*}
where the second equality is based on the property of Martingale~\citep[Proposition~1]{fang2018spider}, the third equality is based on the choice of $q$, the fourth equality is based on the fact $[\xi^t_1,\cdots,\xi^t_n]^\top \sim {\rm Multinomial}(b,\frac{1}{n}\vone)$ and 
\begin{align}\label{eq:E}
    \BE\left[\nabla f_{\xi}(\vx^{t+1}(\xi)) - \nabla f_{\xi}(\vx^t(\xi))\right]=\frac{1}{n}\sum_{i=1}^{n}\big(\nabla f_i(\vx^{t+1}(i)) - \nabla f_i(\vx^t(i))\big),
\end{align}
the first inequality is also based on (\ref{eq:E}), the second inequality is based on the setting of $b\geq (1-p)/p$ and the last step is because of the following inequality.
\begin{align}\label{ieq:diff-grad}
\begin{split}    
    & \frac{1}{n}\sum_{i=1}^n\mathbb{E}\Norm{\nabla f_{i}(\vx^{t+1}(i)) - \nabla f_{i}(\vx^t(i))}^2 \\
\leq &  \frac{3}{n}\sum_{i=1}^n\left(\mathbb{E}\Norm{\nabla f_{i}(\vx^{t+1}(i)) - \nabla f_{i}(\bar\vx^{t+1})}^2 + \mathbb{E}\Norm{\nabla f_{i}(\bar\vx^{t+1}) - \nabla f_{i}(\bar\vx^{t})}^2 + \mathbb{E}\Norm{\nabla f_{i}(\bar\vx^{t}) - \nabla f_{i}(\vx^t(i))}^2 \right)\\  
\leq  &  \frac{3}{n}\sum_{i=1}^n\left(nL^2\mathbb{E}\Norm{\vx^{t+1}(i) - \bar\vx^{t+1}}^2 + \mathbb{E}\Norm{\nabla f_{i}(\bar\vx^{t+1}) - \nabla f_{i}(\bar\vx^{t})}^2 + nL^2\mathbb{E}\Norm{\bar\vx^{t} - \vx^t(i)}^2 \right)\\ 
\leq  &  3L^2\BE\Norm{\mX^{t+1}-\vone\bx^{t+1}}^2 + 3L^2\mathbb{E}\Norm{\bar\vx^{t+1} - \bar\vx^{t}}^2 + 3L^2\BE\Norm{\mX^t-\vone\bx^t}^2 \\
\leq  &  4L^2C^t + 3L^2\mathbb{E}\Norm{\bar\vx^{t+1} - \bar\vx^{t}}^2,
\end{split}
\end{align}
where the second inequality is based on Lemma \ref{lem:smooth}, the third inequality is based on Assumption~\ref{asm:as} and the last step is based on inequalities (\ref{ieq:rho}) and (\ref{ieq:CX}).
   
We finally consider $V^{t+1}=\frac{1}{n}\mathbb{E}\Vert \mG^{t+1} -\nabla\mF(\mX^{t+1})\Vert^2$. The update rule for $\vg_{t+1}(i)$ implies
\begin{align*}
\begin{split}
    V^{t+1}&=\frac{1-p}{n}\sum_{i=1}^n\mathbb{E}\Big\Vert \vg^{t}(i)+\dfrac{\xi_i^t}{bq}\big(\nabla f_i(\vx^{t+1}(i)) - \nabla f_i(\vx^t(i))\big)-\nabla f_i(\vx^{t+1}(i))\Big\Vert^2\\
    &=(1-p)V^t+\frac{1-p}{n}\sum_{i=1}^n\mathbb{E}\Big\Vert \dfrac{\xi_i^t}{bq}\big(\nabla f_i(\vx^{t+1}(i)) - \nabla f_i(\vx^t(i))\big)-\big(\nabla f_i(\vx^{t+1}(i)) - \nabla f_i(\vx^t(i))\big)\Big\Vert^2\\
    &\leq(1-p)V^t+\frac{1-p}{n}\sum_{i=1}^n\dfrac{\Var[\xi_i^t]}{b^2q^2}\BE\Vert \nabla f_i(\vx^{t+1}(i)) - \nabla f_i(\vx^t(i))\Vert^2\\
    &=(1-p)V^t+\frac{1-p}{n}\sum_{i=1}^n\dfrac{bq(1-q)}{b^2q^2}\mathbb{E}\Vert \nabla f_i(\vx^{t+1}(i)) - \nabla f_i(\vx^t(i))\Vert^2\\
    &\leq(1-p)V^t+\frac{1-p}{b}\sum_{i=1}^n\mathbb{E}\Vert\nabla f_i(\vx^{t+1}(i)) - \nabla f_i(\vx^t(i))\Vert^2\\
    &\leq(1-p)V^t+p\sum_{i=1}^n\mathbb{E}\Vert \nabla f_i(\vx^{t+1}(i)) - \nabla f_i(\vx^t(i))\Vert^2\\
    &\leq(1-p)V^t+4pnL^2C^t+3pnL^2\BE\Norm{\bx^{t+1} - \bx^t}^2,
\end{split}
\end{align*}    
where the second equality is based on the property of Martingale~\citep[Proposition~1]{fang2018spider}, the first inequality is based on the fact $\xi_i^t\sim {\rm Binomial}(b,q)$ and the independence between $\xi_i^t$ and $\nabla f_i(\vx^{t+1}(i)) - \nabla f_i(\vx^t(i))$, the second and third and inequalities are based on the settings $q=1/n$ and~$b\geq (1-p)/p$, and the last step is based on the inequality~(\ref{ieq:diff-grad}).
\end{proof}

\subsection{The Proof of Theorem \ref{thm:com}}

\begin{proof}

Recall that Lemma \ref{lem:linear} says
\begin{align*}
        \mathbb{E}[f(\bx^{t+1})-f^*]\leq(1-\mu\eta)\mathbb{E}[f(\bx^{t})-f^*]+\eta U^t+L^2\eta C^t-\left(\frac{1}{2\eta}-\frac{L}{2}\right)\mathbb{E}\Vert\bx^{t+1}-\bx^{t}\Vert^2.
\end{align*}
Combining above inequality with Lemma \ref{lem:CUV}, we have
\begin{align*}
\begin{split}
\Phi^{t+1} 
=& \mathbb{E}[f(\bx^{t+1})-f^*]+\alpha U^{t+1}+\beta V^{t+1}+L C^{t+1}\\
\leq & (1-\mu\eta)\mathbb{E}[f(\bx^{t})-f^*]+\eta U^t+L^2\eta C^t-(\frac{1}{2\eta}-\frac{L}{2})\mathbb{E}\Vert\bx^{t+1}-\bx^{t}\Vert^2\\
 & +\alpha(1-p)U^t+4\alpha pL^2C^t+3\alpha pL^2\BE\Norm{\bx^{t+1} - \bx^t}^2\\
 & +\beta (1-p)V^t+4\beta pnL^2C^t+3\beta pnL^2\BE\Norm{\bx^{t+1} - \bx^t}^2\\
 & +20L \rho^2n^2C^t+4L\rho^2np\eta^2V^t+12L\rho^2n^2\mathbb{E}\Vert\bx^{t+1}-\bx^{t}\Vert^2\\
\leq & (1-\mu\eta)\mathbb{E}[f(\bx^{t})-f^*]+\big(\eta+\alpha(1-p)\big)U^t+\big(\beta(1-p)+4L\rho^2np\eta^2\big)V^t\\
 & +(L^2\eta+4\alpha pL^2+4\beta pnL^2+20L \rho^2n^2)C^t\\
 & -\left(\frac{1}{2\eta}-\frac{L}{2}-3\alpha pL^2-3\beta pnL^2-12L\rho^2n^2\right)\mathbb{E}\Vert\bx^{t+1}-\bx^{t}\Vert^2.\\
\end{split}
\end{align*}
The setting of $p,\eta,b,\rho,\alpha$ and $\beta$ implies
\begin{align*}
     \eta+\alpha(1-p) & \leq \alpha(1-\mu\eta), \\
     \beta(1-p)+4L\rho^2np\eta^2 & \leq \beta(1-\mu\eta), \\
     L^2\eta+4\alpha pL^2+4\beta pnL^2+20L \rho^2n^2 & \leq L(1-\mu\eta),
\end{align*}
and
\begin{align*}
    \frac{1}{2\eta}-\frac{L}{2}-3\alpha pL^2-3\beta pnL^2-12L\rho^2n^2 \geq 0.
\end{align*}

Therefore, we have
\begin{align*}
    \Phi^{t+1}\leq(1-\mu\eta)\Phi^t,
\end{align*}
and the setting
\begin{equation*}
    T\geq \left\lceil\frac{1}{\mu\eta}\log\frac{\Phi^0}{\epsilon}\right\rceil
\end{equation*}
leads to
\begin{align}\label{ieq:PhiT}
     \Phi^T \leq (1-\mu\eta)^T\Phi^0 \leq \epsilon.
\end{align}

The output $\vx^{\rm out}$ holds that
\begin{align*}
\BE[f(\vx^{\rm out})-f^*]
&=\BE[f(\bx^T)-f^*]+\BE[f(\vx^{\rm out})-f(\bx^T)]\\
&=\BE[f(\bx^T)-f^*]+\frac{1}{n}\sum_{i=1}^n\BE[f(\vx^T(i))-f(\bx^T)]\\
&\leq\BE[f(\bx^T)-f^*]+\frac{1}{n}\sum_{i=1}^n\left(\BE\left[\langle\nabla f(\bx^T),\vx^T(i)-\bx^T\rangle + \frac{L}{2}\BE\Norm{\vx^T(i)-\bx^T}^2\right]\right)\\
&=\BE[f(\bx^T)-f^*]+\frac{L}{2n}\sum_{i=1}^n\BE\Norm{\vx^T(i)-\bx^T}^2\\
&\leq\BE[f(\bx^T)-f^*]+\frac{L}{2n}C^T\\
&\leq \Phi^T \leq \epsilon,
\end{align*}
where the first inequality is based on Lemma \ref{lem:smooth}, the second inequality is based on the definition of $C^t$,  the third inequality is based on the definition of $\Phi^t$ and the last step is based on inequality (\ref{ieq:PhiT}).
\end{proof}

\subsection{The Proof of Corollary \ref{cor:decentralized-upper}}

\begin{proof}
The parameters setting in this corollary means
\begin{align*}
\begin{split}    
& p=\Theta\left(\max\left\{\frac{1}{\sqrt{n}},\frac{1}{\kappa}\right\}\right), \qquad 
b=\Theta\left(\min\{\sqrt{n},\kappa\}\right
), \qquad \eta=\Theta\left(\frac{1}{L}\right),    \\
& K=\Theta\left(\frac{\sqrt{2}\,(4+\log n)}{(\sqrt{2}-1)\sqrt{\gamma}}\right) \qquad \text{and}\qquad     T=\Theta\left(\kappa\log\left(\frac{1}{\epsilon}\right)\right).
\end{split}
\end{align*}

At each iteration, the algorithm takes 
$pn+(1-p)b=\fO\big(n/\min\{\kappa,\sqrt{n}\}\big)$ LFO calls (in expectation),
$K$ communication rounds and $1+K\tau$ time cost.
Multiplying the overall iteration numbers $T$ on these quantities completes the proof.
\end{proof}

\begin{remark}
We omit the detailed proof of Theorem \ref{thm:DGD-GT}, since it can be easily achieved by following Theorem \ref{thm:com} and Corollary~\ref{cor:decentralized-upper} with $p=1$, $U^t=V^t=0$ and showing the linear convergence of $\BE[f(\bx^t)-f^*] + LC^t$.
\end{remark}

\section{More Details for Experiments}\label{appendix:experiments}

All of our experiments are performed on PC with Intel(R) Core(TM) i7-8550U CPU@1.80GHz processor and we implement the algorithms by MPI for Python 3.9. 

We formally present the details of centralized gradient descent (CGD) in Algorithm \ref{alg:CGD}. The network in CGD has sever-client architecture, which allows the server to communicate with all clients. 
The convergence of CGD can be described by gradient descent step
\begin{align*}
    \tilde\vx^{t+1} = \tilde\vx^{t} - \eta\nabla f(\tilde\vx^{t}),
\end{align*}
which can find an $\epsilon$-suboptimal solution within $\fO(\kappa\log(1/\epsilon))$ iterations~\cite{karimi2016linear}.
Therefore, it requires the LFO complexity of $\fO(\kappa n\log(1/\epsilon))$.
Since the sever-client architecture does not suffer from consensus error, it has communication complexity of $\fO\left(\kappa\log(1/\epsilon)\right)$ and time complexity of 
$\fO\left(\kappa\left(1+{\tau}\right)\log(1/\epsilon)\right)$.

The objective functions for linear regression and logistic regression are not strongly convex when $d>m$, while they satisfy the PL conditions. Please refer to Section 3.2 of \citet{karimi2016linear}.

\begin{algorithm}[t]
\caption{Centralized GD}\label{alg:CGD}
\begin{algorithmic}[1]\vskip0.05cm
\STATE \textbf{Input:} initial point $\tilde\vx^0\in\BR^d$, iteration number $T$ and stepsize $\eta>0$  \\[0.12cm]
\STATE \textbf{for $t=0,1\cdots T-1$ do} \\[0.12cm]
\STATE \quad \textbf{for} $i=1,\dots,n$ \textbf{do in parallel}  \\[0.12cm]
\STATE \quad\quad  \client \\[0.12cm]
\STATE \quad \quad \quad receive $\tilde\vx^t$ \\[0.12cm] 
\STATE \quad \quad \quad compute $\vg^t(i)=\nabla f_i(\vx^t)$ \\[0.12cm] 
\STATE \quad \quad \quad send $\vg^t(i)$ \\[0.12cm] 
\STATE \quad \textbf{end for}\\[0.05cm]
\STATE \quad \server \\[0.12cm]
\STATE \quad\quad receive $\vg^t(1),\dots,\vg^t(n)$ \\[0.12cm]
\STATE \quad \quad compute $\tilde\vx^{t+1}=\tilde\vx^t - \eta\cdot\frac{1}{n}\sum_{i=1}^n\vg^t(i)$ \\[0.12cm] 
\STATE \quad \quad broadcast $\tilde\vx^{t+1}$ \\[0.12cm] 
\STATE \textbf{end for}\\[0.04cm]
\STATE \textbf{Output:} $\tilde\vx^T$  \\[0.05cm]
\end{algorithmic}
\end{algorithm}

\bibliographystyle{plainnat}
\bibliography{reference}
\end{document}

%% file: math.tex
\usepackage{amsmath}\allowdisplaybreaks
\usepackage{amsfonts,bm}
\usepackage{amssymb}

\def\eqref#1{equation~(\ref{#1})}

\def\1{\bf{1}}

\newcommand{\Norm}[1]{\left\| #1 \right\|}
\newcommand{\norm}[1]{\left\| #1 \right\|_2}
\def\inner#1#2{\langle #1, #2 \rangle}

\def\vzero{{\bf{0}}}
\def\vone{{\bf{1}}}

\def\va{{\bf{a}}}
\def\vb{{\bf{b}}}

\def\ve{{\bf{e}}}

\def\vg{{\bf{g}}}

\def\vs{{\bf{s}}}

\def\vu{{\bf{u}}}
\def\vv{{\bf{v}}}

\def\vx{{\bf{x}}}
\def\vy{{\bf{y}}}

\def\fC{{\mathcal{C}}}

\def\fE{{\mathcal{E}}}

\def\fG{{\mathcal{G}}}

\def\fI{{\mathcal{I}}}

\def\fM{{\mathcal{M}}}

\def\fO{{\mathcal{O}}}

\def\fV{{\mathcal{V}}}

\def\BE{{\mathbb{E}}}

\def\BI{{\mathbb{I}}}

\def\BN{{\mathbb{N}}}

\def\BR{{\mathbb{R}}}

\def\mA {{\bf A}}

\def\mF {{\bf F}}
\def\mG {{\bf G}}

\def\mI {{\bf I}}

\def\mR {{\bf R}}
\def\mS {{\bf S}}

\def\mU {{\bf U}}

\def\mW {{\bf W}}
\def\mX {{\bf X}}
\def\mY {{\bf Y}}

\newcommand{\Var}{\mathrm{Var}}

\usepackage{amsthm}
\usepackage{etoolbox}

\makeatletter
\def\Ddots{\mathinner{\mkern1mu\raise\p@
\vbox{\kern7\p@\hbox{.}}\mkern2mu
\raise4\p@\hbox{.}\mkern2mu\raise7\p@\hbox{.}\mkern1mu}}
\makeatother

\makeatletter
\newcommand*{\rom}[1]{\expandafter\@slowromancap\romannumeral #1@}
\makeatother

\def\FM {{{\rm AccGossip}}}
\def\bg {{{\bar \vg}}}
\def\bs {{{\bar \vs}}}

\def\bx {{{\bar \vx}}}

\def\PL {{Polyak--{\L}ojasiewicz}}

\def\server{\underline{server}:}
\def\client{\underline{$i$-th client}:}